\documentclass[11pt]{article}
\usepackage{amsmath, amssymb, amsthm, amsfonts}
\usepackage{indentfirst}
\usepackage[mathscr]{euscript}
\usepackage{amscd}
\usepackage{color}
\usepackage[pdftex]{graphicx}

\numberwithin{equation}{section}

\theoremstyle{plain}
\newtheorem{theorem}{Theorem}\numberwithin{theorem}{section}
\newtheorem{lemma}{Lemma}\numberwithin{lemma}{section}
\newtheorem{proposition}{Proposition}\numberwithin{proposition}{section}
\newtheorem{corollary}{Corollary}\numberwithin{corollary}{section}
\theoremstyle{definition}
\newtheorem{definition}{Definition}\numberwithin{definition}{section}
\theoremstyle{remark}
\newtheorem{remark}{Remark}\numberwithin{remark}{section}

\newcommand{\R}{\mathbb{R}}

\newcommand{\F}{{}_pF_q}

\title{The Newton polyhedron and positivity \\of ${}_2F_3$ hypergeometric functions}

\author{Yong-Kum Cho\footnote{ykcho@cau.ac.kr.
Department of Mathematics, College of Natural Sciences, Chung-Ang University,
84 Heukseok-Ro, Dongjak-Gu, Seoul 06974, Korea.}
\and Seok-Young Chung\footnote{sychung@knights.ucf.edu. Department of Mathematics, University of Central Florida, 
4393 Andromeda Loop N, Orlando, FL 32816, USA.}}

\date{}

\begin{document}

\maketitle

\textsc{Abstract.} {\small As for the ${}_2F_3$ hypergeometric function of the form
\begin{equation*}
{}_2F_3\left[\begin{array}{c}
a_1, a_2\\ b_1, b_2, b_3\end{array}\biggr| -x^2\right]\qquad(x>0),
\end{equation*}
where all of parameters are assumed to be positive, we give sufficient conditions
on $(b_1, b_2, b_3)$ for its positivity in terms of Newton polyhedra with vertices
consisting of permutations of $\,(a_2, a_1+1/2, 2a_1)\,$ or $\,(a_1, a_2+1/2, 2a_2).$
As an application, we obtain an extensive validity region of $(\alpha, \lambda, \mu)$
for the inequality
\begin{equation*}
\int_0^x (x-t)^{\lambda}\, t^{\mu} J_\alpha(t)\, dt \ge 0\qquad(x>0).
\end{equation*}}

\medskip


{\small ${}_pF_q$ hypergeometric functions, sums of squares method, transference principle.}

\medskip


\section{Introduction}
This paper concerns primarily the geometric structure of parameters
that ensures the positivity of ${}_2F_3$ hypergeometric functions of type
\begin{equation}\label{TF}
\Phi(x) \equiv {}_2F_3\left[\begin{array}{c}
a_1, a_2\\ b_1, b_2, b_3\end{array}\biggr| -x^2\right]\qquad(x>0),
\end{equation}
where all of parameters $a_i, b_j$ are assumed to be positive.

In our previous studies \cite{CCY2}, \cite{CY}, we considered the problem of positivity
for ${}_1F_2$ hypergeometric functions of similar type and described regions of positivity
in the plane of denominator-parameter pairs by means of Newton polyhedra and their hyperbolic
extensions (see Theorem \ref{theoremN} for a summary).
In a consistent manner, we aim at providing sufficient conditions for the positivity of
${}_2F_3$ hypergeometric functions of type \eqref{TF} in terms of
Newton polyhedra in the $(b_1, b_2, b_3)$-space for each fixed $(a_1, a_2)$.

To state briefly, let $\mathcal{A}, \mathcal{B}$ be the spatial sets of all permutations of
$$\,(a_2, a_1+1/2, 2a_1), \quad (a_1, a_2+1/2, 2a_2),$$
respectively. In view of the known formula (\cite[6.2]{L}, \cite[5.4]{Wa})
\begin{align}\label{J}
J_\nu^2(x) = \frac{(x/2)^{2\nu}}{\left[\Gamma(\nu+1)\right]^2}\,
{}_1 F_{2}\left[\begin{array}{c} \nu +1/2\\ \nu+1, 2\nu +1\end{array}
\biggr| \,-x^2\right],
\end{align}
where $J_\nu$ stands for the Bessel function of the first kind of order $\,\nu\in\R,$
it is evident that $\Phi$ is nonnegative for each $\,\mathbf{b} = (b_1, b_2, b_3)\in \mathcal{A}\cup\mathcal{B}.$

By means of fractional integrals, to be explained in detail, it is simple to observe
a transference principle which asserts if $\Phi$ is nonnegative for a point $\mathbf{b}$,
then $\Phi$ remains strictly positive for all points of the octant
$\,\mathbf{b} +\R_+^3\,$ except for the corner $\mathbf{b}$. On applying this transference principle,
it is thus found that $\Phi$ remains strictly positive for all points of
\begin{equation}\label{TF2}
\bigcup_{\mathbf{b}\,\in\,\mathcal{A}\cup\mathcal{B}} \left( \mathbf{b} +\R_+^3\right)
\end{equation}
excluding $\mathcal{A}\cup\mathcal{B}$ where $\Phi$ is shown to be nonnegative.

As it is natural to ask if this positivity region could be extended further,
we shall prove that $\Phi$ indeed remains positive for all points of the convex hull
containing \eqref{TF2}, so called the Newton polyhedron of $\mathcal{A}\cup\mathcal{B},$
under certain \emph{admissible} conditions on $a_1, a_2$. From a geometric view-point,
the Newton polyhedron of $\mathcal{A}\cup\mathcal{B}$ represents an infinite polygonal region
in $\R_+^3$ surrounded by a hexagonal or triangular bottom face, several side
quadrilaterals and infinite planar faces parallel to one of the coordinate axes.

An expansion formula due to Gasper \cite{Ga1} gives
\begin{equation}
\Phi(x) = x^{-2\nu}\sum_{n=0}^\infty C_n(\nu)\, J_{n+\nu}^2(x)
\end{equation}
for any real number $\nu$ subject to certain condition. Once each coefficient $C_n(\nu)$
were shown to be nonnegative, we may conclude that $\Phi$ is nonnegative from this expansion.
Often referred to as \emph{Gasper's sums of squares method}, it is this framework that our investigation
will be based on.

In practice, as each coefficient $C_n(\nu)$ involves either terminating
${}_4F_3$ or ${}_5F_4$ hypergeometric series, we are confronted with
the problem of how to determine the sign of those
terminating series. In particular, if $(b_1, b_2, b_3)$ belongs to one of
the aforementioned side quadrilaterals, we must deal with the terminating
${}_5F_4$ hypergeometric series of the from
\begin{equation}\label{TF4}
{}_{5}F_{4} \left[\begin{array}{c} -n, n+\alpha_1, \alpha_2, \alpha_3, \alpha_4\\
\beta_1, \beta_2, \beta_{3}, \beta_4\end{array}\right],\quad n = 1, 2, \cdots,
\end{equation}
with appropriate values of $\,\alpha_i, \beta_j\,$ expressible in terms of $\,a_i, b_j, \nu.$

In our previous work \cite{CCY2}, we exploited one of Whipple's transformation formulas and an
induction argument to set up a criterion of positivity for the terminating
${}_4F_3$ hypergeometric series of similar type. Concerning \eqref{TF4}, however,
as an appropriate version of Whipple's transformation formula
is no longer available, we shall make use of a modification of more general transformation formulas
developed by Fields and Wimp \cite{FW} to decompose it into a finite sum involving
products of terminating ${}_4F_3$ hypergeometric series,
which will enable us to obtain a set of positivity criteria.

In summary, it turns out that each of Newton polyhedra with vertices
$\,\mathcal{A}, \,\mathcal{B}, \,\mathcal{A}\cup\mathcal{B}\,$
is available as a positivity region of $\Phi$ under different conditions on the pair $(a_1, a_2)$.
For this reason, our positivity results will be presented
in three separate statements (Theorems \ref{theorem2}, \ref{theorem3}, \ref{theorem4})
in each of which an analytic description of the associated Newton polyhedron will be given
for the sake of practical applications.

The problem of positivity for the ${}_2F_3$ hypergeometric function of type \eqref{TF}
arises often as a critical issue in various disguises. As an illustration, the famous problem of
determining $\alpha, \lambda, \mu$ for validity of
\begin{equation}\label{TF5}
\int_0^x (x-t)^{\lambda}\, t^{\mu} J_\alpha(t)\, dt\, \ge \,0\qquad(x>0)
\end{equation}
is equivalent to the problem of positivity for a special case of \eqref{TF}
implicitly involved (see \eqref{F5} for details). In application of our results,
we shall specify an extensive validity region of $(\alpha, \lambda, \mu)$ for the inequality \eqref{TF5} which will provide
at once simplified proofs or improvements for a number of known inequalities and answers to some open conjectures.

As usual, we shall define the ${}_pF_q$ hypergeometric function by
\begin{equation}\label{pFq}
{}_pF_q\left[\begin{array}{c} u_1, \cdots, u_p\\
v_1, \cdots, v_q\end{array}\biggr| \,z\right]
= \sum_{n=0}^\infty \frac{(u_1)_n\cdots (u_p)_n}
{n!\,(v_1)_n\cdots (v_q)_n}\, z^n\qquad(z\in\mathbb{C}),
\end{equation}
where parameters $u_i, v_j$ are real numbers subject to the condition that none of
$v_j$'s coincides with a negative integer and the coefficients are written in Pochhammer's notation;
for $\,\alpha\in\R\,$ and a positive integer $n$,
$$(\alpha)_n = \alpha(\alpha+1)\cdots(\alpha+n-1),
\quad (\alpha)_0=1.$$
In the special case $\,z=1,$ it is customary to delete the unit argument.
For the general reference, we refer to Bailey \cite{Ba} and Luke \cite{L}.

To simplify notation, we shall use throughout the logical symbol $\wedge$ to denote \emph{and};
$A\wedge B$ is true if and only if both $A$ and $B$ are true. In addition, we shall write
$\,\R_+ = [0, \infty)\,$ and $\,\R_+^* = (0, \infty)\,$ distinctively.

\section{Transference principle and necessities}
As observed by many authors in different contexts (see \cite{As2}, \cite{CY}, \cite{FI}, \cite{Ga1}),
it is possible to transfer a known positivity region to a much larger region
with the aid of fractional integrals. We shall denote the beta function by
$$B(\alpha, \beta) = \int_0^1 (1-t)^{\alpha-1} t^{\beta-1} dt \qquad(\alpha>0, \,\beta>0).$$

\begin{proposition}\label{proposition1} {\rm (transference principle)}
Suppose
\begin{equation*}
\Phi(\mathbf{a}, \mathbf{b}; x) \equiv {}_2F_3\left[\begin{array}{c}
a_1, a_2\\ b_1, b_2, b_3\end{array}\biggr| -x^2\right]\ge 0 \qquad (x>0)
\end{equation*}
for some $\,\mathbf{a} = (a_1, a_2)\in\R^2\,$ and $\,\mathbf{b} = (b_1, b_2, b_3)\in \R^3,\,$ where
all of the components of $\,\mathbf{a}, \mathbf{b}\,$ are assumed to be positive.

\begin{itemize}
\item[\rm(i)] For any $\,\boldsymbol{\epsilon} = (\epsilon_1, \epsilon_2, \epsilon_3)\in\R_+^3,\,$ if $\,
 \epsilon_1, \epsilon_2, \epsilon_3\,$ are not simultaneously zero, then
$\,\Phi(\mathbf{a}, \mathbf{b}+\boldsymbol{\epsilon}; x)>0\,$ for all $\,x>0.$
\item[\rm(ii)] For any $\,\boldsymbol{\delta} = (\delta_1, \delta_2)\in \R_+^2,\,$
if $\,\delta_1, \delta_2\,$ are not simultaneously zero and $\,0\le\delta_j<a_j,\,j=1, 2,\,$
then $\,\Phi(\mathbf{a}-\boldsymbol{\delta}, \mathbf{b}; x)>0\,$ for all $\,x>0.$
\end{itemize}
\end{proposition}

\begin{proof}
For $\,\epsilon_1>0,$ the readily-verified representation
\begin{align*}
\Phi(\mathbf{a}, b_1 +\epsilon_1, b_2, b_3; x)
= \frac{2}{B(b_1, \epsilon_1)} \int_0^1 \Phi(\mathbf{a}, \mathbf{b}; xt)(1-t^2)^{\epsilon_1-1} t^{2b_1-1}dt
\end{align*}
implies $\,\Phi(\mathbf{a}, b_1 +\epsilon_1, b_2, b_3; x)>0\,$ for all $\,x>0.$ By permuting
the role of $\,b_j, \epsilon_j\,$ or
replacing the kernel appropriately, it is straightforward to verify part (i) in all
possible occasions.

In a similar manner, if $\,0<\delta_1<a_1,\,$ then the representation
\begin{align*}
\Phi(a_1-\delta_1, a_2, \mathbf{b}; x)
&= \frac{2}{B(\delta_1, a_1-\delta_1)} \int_0^1 \Phi(\mathbf{a}, \mathbf{b}; xt)(1-t^2)^{\delta_1-1} t^{2a_1-2\delta_1-1}dt
\end{align*}
implies $\,\Phi(a_1-\delta_1, a_2, \mathbf{b}; x)>0\,$ for all $\,x>0.\,$
By permuting the role of $\,a_j, \delta_j\,$ or
replacing the kernel appropriately, it is also straightforward to verify part (ii)
in all possible occasions.
\end{proof}

By making use of the asymptotic behavior and the above transference principle,
we deduce a set of necessary conditions as follows.

\begin{proposition}\label{proposition2} For $\,0<a_1\le a_2\,$ and $\,b_j>0,\,j=1, 2, 3,$ put
\begin{equation*}
\Phi(x) \equiv {}_2F_3\left[\begin{array}{c}
a_1, a_2\\ b_1, b_2, b_3\end{array}\biggr| - x^2\right] \qquad(x>0).
\end{equation*}
If $\,\Phi(x)\ge 0\,$ for all $\,x>0,\,$ then it is necessary that
\begin{equation}\label{NC}
b_1\wedge b_2\wedge b_3 \ge a_1, \quad b_1 + b_2 + b_3 \ge 3a_1 + a_2 + 1/2.
\end{equation}
Moreover, $\Phi$ changes sign at least once if these conditions are violated.
\end{proposition}

\begin{proof}
Assume first that $a_2 - a_1$ is not an integer. Setting
$$\chi = b_1+b_2+b_3-a_1-a_2-1/2,$$
a well-known asymptotic expansion formula (\cite{L}, \cite{O}) gives
\begin{align}\label{AS}
\Phi(x) &= \frac{\,\Gamma(b_1)\Gamma(b_2)\Gamma(b_3)\,}{\Gamma(a_1)\Gamma(a_2)}\nonumber\\
&\quad\,\,\,\biggl\{\frac{\Gamma(a_1)\Gamma(a_2-a_1)}{\,\Gamma(b_1-a_1)\Gamma(b_2-a_1)\Gamma(b_3-a_1)\,}
\left(\frac{x}{2}\right)^{-2a_1}\left[1+O\left(x^{-2}\right)\right] \nonumber\\
&\quad + \frac{\Gamma(a_2)\Gamma(a_1-a_2)}{\,\Gamma(b_1-a_2)\Gamma(b_2-a_2)\Gamma(b_3-a_2)\,}
\left(\frac{x}{2}\right)^{-2a_2}\left[1+O\left(x^{-2}\right)\right]\nonumber\\
&\quad +\frac{1}{\,\sqrt{\pi}\,}\left(\frac{x}{2}\right)^{-\chi}
\left[\cos\left(x-\frac{\pi\chi}{2}\right)+O\left(x^{-1}\right)\right]\biggr\}
\end{align}
as $x\to\infty$. In view of the oscillatory nature of the last part,
it is evident that the condition $\,\chi\ge 2\min (a_1,\,a_2)\,$ or equivalently
\begin{equation}\label{NC2}
b_1 + b_2 + b_3 \ge 3a_1 + a_2 + 1/2
\end{equation}
is necessary. Otherwise, $\Phi$ oscillates in sign infinitely often.

Suppose now $a_2-a_1$ happens to be an integer. For $\,0<\delta<\min (1,\,a_2),\,$ it follows from the
transference principle of Proposition \ref{proposition1} that
\begin{equation*}
{}_2F_3\left[\begin{array}{c}
a_1, a_2-\delta \\ b_1, b_2, b_3\end{array}\biggr| -x^2\right]>0\qquad(x>0)
\end{equation*}
and hence $b_1, b_2, b_3$ must satisfy the condition
\begin{equation*}
b_1 + b_2 + b_3 \ge a_1 + a_2 -\delta + 1/2 + 2 \min(a_1, \,a_2-\delta).
\end{equation*}
By letting $\,\delta\to 0,\,$ it is thus found that \eqref{NC2} is necessary.

It remains to verify the additional requirements $\,b_j\ge a_1\,$ for all $j$ under the
assumption \eqref{NC2}. By using a limiting argument as above, it suffices to deal with the case
$\,b_1+b_2+b_3>3a_1+a_2+1/2.$ Let us assume $\,0<b_1<a_1.$ Integrating termwise, it is
elementary to deduce
\begin{align*}
{}_1F_2\left[\begin{array}{c}
a_2\\ b_2, b_3\end{array}\biggr| -x^2\right] =
\frac{2}{B(b_1, a_1-b_1)}\int_0^1 \Phi(xt) (1-t^2)^{a_1-b_1 -1} t^{2b_1-1} dt >0
\end{align*}
for all $\,x>0,$ whence $\,b_2\wedge b_3> a_2\,$ according to Theorem \ref{theoremN}.

In the case $\,a_1< a_2,$ the asymptotic behavior \eqref{AS} implies
\begin{equation*}
\Phi(x) \sim  \frac{\,\Gamma(b_1)\Gamma(b_2)\Gamma(b_3)\Gamma(a_2-a_1)\,}
{\Gamma(a_2)\Gamma(b_1-a_1)\Gamma(b_2-a_1)\Gamma(b_3-a_1)\,}
\left(\frac{x}{2}\right)^{-2a_1}
\end{equation*}
as $\,x\to\infty.$ Replacing $b_1$ by $b_1+\epsilon$ with $\,\epsilon>0,$ if necessary,
we may assume $\,a_1-1<b_1<a_1\,$ so that $\,\Gamma(b_1-a_1)<0\,$ and the multiplicative constant
is negative. Since $\Phi$ is smooth with $\,\Phi(0) =1,$ it implies that $\Phi$ must change sign at least once,
which contradicts the nonnegativity assumption on $\Phi$. Thus $\,b_1\ge a_1.$
In the case $\,a_1 = a_2,\,$ replacing $a_1$ by
$\,a_1-\delta,\,0<\delta<a_1,\,$ the same reasonings would lead to $\,b_1\ge a_1 -\delta \,$ and subsequently
$\,b_1\ge a_1\,$ by letting $\,\delta\to 0.\,$ By symmetry, it is also necessary that $\,b_2\ge a_1,\,b_3\ge a_1.\,$
\end{proof}

\section{Expansions of ${}_pF_q$ hypergeometric functions}
The purpose of this section is to give a list of expansion formulas for
${}_pF_q$ hypergeometric functions which will be applied frequently in the sequel.

For simplicity, it is often convenient to use the vector notation
$$\boldsymbol{u}_p = (u_1, \cdots, u_p),\quad (\boldsymbol{u}_p)_n = (u_1)_n\cdots (u_p)_n$$
so that the ${}_pF_q$ hypergeometric function of \eqref{pFq} can be written as
\begin{equation*}
    \F\left[\begin{array}{c} \boldsymbol{u}_p \\
    \boldsymbol{v}_q \end{array}\biggr| z\right]
    =\sum_{n=0}^\infty\frac{(\boldsymbol{u}_p)_n}{n!\,(\boldsymbol{v}_q)_n} z^n.
\end{equation*}

An expansion formula due to Fields and Wimp \cite[3.3]{FW} states
\begin{align}\label{FW}
\F\left[\begin{array}{c} \boldsymbol{u}_p \\
    \boldsymbol{v}_q \end{array}\biggr| zw\right]
&=\sum_{n=0}^\infty \frac{(\boldsymbol{\alpha}_r)_n (-z)^n}{n! (n+\delta)_n (\boldsymbol{\beta}_s)_n}
\,{}_r F_{s+1}\left[\begin{array}{c} \boldsymbol{n+\alpha}_r\\ 2n + \delta +1, \boldsymbol{n+\beta}_s
\end{array}\biggr| z\right]\nonumber\\
&\qquad\times\, {}_{p+s+2}F_{q+r}\left[\begin{array}{c} -n, n+\delta, \boldsymbol{\beta}_s, \boldsymbol{u}_p\\
\boldsymbol{\alpha}_r, \boldsymbol{v}_q\end{array}\biggr| w\right],
\end{align}
valid for any nonnegative integers $\,p, q, r, s,\,$ real parameters $\,\boldsymbol{u}_p, \boldsymbol{v}_q,
\boldsymbol{\alpha}_r, \boldsymbol{\beta}_s, \delta,\,$ and complex arguments $z, w$, provided that
none of $\,\boldsymbol{v}_q, \boldsymbol{\alpha}_r, \boldsymbol{\beta}_s, \delta\,$ coincides with a negative integer
and each of the involved series converges. We refer to Luke and Coleman \cite{LW} for closely related
expansion formulas.

In view of the formula \eqref{J}, if we choose $\,r=s=1\,$ and take
$$\alpha_1=\nu+1/2, \,\,\beta_1=\nu+1, \,\,\delta = 2\nu,\,$$
then \eqref{FW} yields Gasper's \emph{sums of squares formula} \cite[(3.1)]{Ga1}
\begin{align}\label{G}
\F\left[\begin{array}{c} \boldsymbol{u}_p\\
    \boldsymbol{v}_q \end{array}\biggr|\,-z^2\right]
    &=\Gamma^2(\nu+1)\left(\frac z2\right)^{-2\nu}\sum_{n=0}^\infty  \frac{2n+2\nu}{n+2\nu} \frac{(2\nu+1)_n}{n!}\nonumber\\
&\times\,{}_{p+3}F_{q+1}\left[\begin{array}{c} -n, n+2\nu, \nu+1, \boldsymbol{u}_p\\
\nu +1/2, \boldsymbol{v}_q\end{array}\right]\,J^2_{n+\nu}\left(z\right),
\end{align}
valid if $\,p\le q\,$ and none of $\,2\nu, \boldsymbol{v}_q\,$ coincides with a negative integer.

In the special case when one of numerator parameters happens to be a negative integer,
the expansion formula \eqref{FW} gives rise to various kinds of transformation formulas for
terminating hypergeometric series with unit argument. Of subsequent importance will be the following formula:

\begin{itemize}
\item[{}]{\it For each positive integer $n$, if $\,p\ge 3,\,$ then
    \begin{align}\label{T}
   & {}_{p+1}F_p\left[\begin{array}{c} -n, u_1, \cdots, u_p\\
        v_1, \cdots, v_p\end{array}\right]\nonumber\\
         &\quad =\sum_{k=0}^n\binom{n}{k}
        \frac{(\sigma)_k (u_3)_k\cdots(u_p)_k}{(k+\delta)_k (v_3)_k\cdots(v_p)_k}
        \,{}_4F_3\left[\begin{array}{c} -k, k+\delta, u_1, u_2\\
        \sigma, v_1, v_2\end{array}\right]\nonumber\\
        &\qquad\quad\times\,\,{}_pF_{p-1}\left[\begin{array}{c} -n+k, k+\sigma, k+u_3, \cdots, k+u_p\\
        2k+\delta+1, k+v_3, \cdots, k+v_p\end{array}\right].
        \end{align}}
   \end{itemize}

The proof is immediate on taking $\,r = p-1,\,s=p-3\,$ and
$$\boldsymbol{\alpha}_{r} = (-n, \sigma, u_3, \cdots, u_p),\,\,
\boldsymbol{\beta}_{s} = (v_3, \cdots, v_p), \,\, z=w=1.$$
An advantageous feature in this transformation formula is that
there are two {\it free parameters} $\delta, \sigma$ available. For instance, if we recall \cite [2.2(1)] {Ba}
\begin{equation}\label{S}
{}_3F_2\left[\begin{array}{c} -k, k+\alpha_1, \alpha_2\\\beta_1, \beta_2\end{array}\right]
= \frac{(\beta_1 -\alpha_2)_k (\beta_2-\alpha_2)_k}{(\beta_1)_k (\beta_2)_k},
\end{equation}
valid under the Saalsch\"utzian condition $\,1+\alpha_1 +\alpha_2 = \beta_1 +\beta_2,\,$
and choose $\,\sigma = u_2,\, \delta = v_1 + v_2 -u_1-1,\,$ then formula \eqref{T} yields
\begin{align}\label{G1}
   &{}_{p+1}F_p\left[\begin{array}{c} -n, u_1, \cdots, u_p\\
        v_1, \cdots, v_p\end{array}\right]\nonumber\\
         &\quad =\sum_{k=0}^n\binom{n}{k}
        \frac{(v_1-u_1)_k (v_2-u_1)_k(u_2)_k\cdots(u_p)_k}
                {(k+v_1+v_2-u_1-1)_{k}(v_1)_k\cdots(v_p)_k}\nonumber\\
        &\qquad\quad\times\,\,{}_pF_{p-1}\left[\begin{array}{c} -n+k, k+u_2, k+u_3, \cdots, k+u_p\\
        2k +v_1+v_2-u_1, k+v_3, \cdots, k+v_p\end{array}\right],
        \end{align}
which coincides with Gasper's transformation formula \cite[(5.1)]{Ga2}.

\section{Positivity of terminating series}
This section focuses on establishing positivity criteria for
terminating ${}_4F_3$ or ${}_5F_4$ hypergeometric series
under the Saalsch\"utzian condition, which will play crucial roles in later developments.

\subsection{Terminating ${}_4F_3$ hypergeometric series}
\begin{lemma}\label{lemmaA1} For each positive integer $n$, put
\begin{equation*}
\Theta_n={}_{4}F_{3} \left[\begin{array}{c} -n, n+\alpha_1, \alpha_2, \alpha_3\\
\beta_1, \beta_2, \beta_{3}\end{array}\right]
\end{equation*}
and assume that $\,\alpha_1>-1,\,\alpha_2>0,\,\alpha_3>0\,$ and
\begin{equation}\label{A1}
1+\alpha_1+\alpha_2+\alpha_3 = \beta_1+\beta_2+\beta_3.
\end{equation}
If parameters $\alpha_j, \beta_j$ satisfy either set of the following
additional conditions simultaneously, then $\,\Theta_n \ge0\,$ for all $\,n\ge 1.\,$
\begin{align*}
{\rm (A1)}\,\left\{\begin{aligned}
&{\,\alpha_2 \le\beta_3\le 1+\alpha_1,}\\
&{\,\underline{\,\alpha_3 \le \beta_1\wedge \beta_2}.}\end{aligned}\right.
\quad {\rm (A2)}\,\left\{\begin{aligned}
&{\,\alpha_2 \le\beta_3\le 2+\alpha_1,}\\
&{\,\underline{\,\alpha_3 \le \beta_1\wedge\beta_2},}\\
&{\,(1+\alpha_1)\alpha_2\alpha_3 \le \beta_1\beta_2\beta_3.}\end{aligned}\right.
\end{align*}
Moreover, the underlined condition may be replaced by
\begin{equation}\label{A4}
\beta_1\wedge\beta_2>0, \quad \alpha_3-1\le\beta_1\wedge\beta_2\le\alpha_3.
\end{equation}
\end{lemma}

\begin{proof} Under Saalsch\"utzian condition \eqref{A1}, if we apply \eqref{G1} and
simplify coefficients with the aid of Saalsch\"utz's formula \eqref{S}, then we find
\begin{align}\label{G2}
   \Theta_n  = \frac{1}{(\beta_3)_n}\sum_{k=0}^n &\binom{n}{k}
        \frac{(\beta_1-\alpha_3)_k (\beta_2-\alpha_3)_k}{ (k+\alpha_1 + \alpha_2 -\beta_3)_{k}}\frac{(n+\alpha_1)_k(\alpha_2)_k}
                {(\beta_1)_k(\beta_2)_k}\nonumber\\
        &\quad\times\,\,\frac{(k+1+\alpha_1-\beta_3)_{n-k} (\beta_3-\alpha_2)_{n-k}}{(2k+1+\alpha_1+\alpha_2-\beta_3)_{n-k}}.
          \end{align}

Since coefficients are all nonnegative under stated conditions of (A1),
the nonnegativity of $\Theta_n$ is obvious. On writing
\begin{align}\label{A41}
&(\beta_1-\alpha_3)_k (\beta_2-\alpha_3)_k\nonumber\\
&\qquad = (\beta_1-\alpha_3)(\beta_2-\alpha_3)(\beta_1-\alpha_3+1)_{k-1} (\beta_2-\alpha_3+1)_{k-1}
\end{align}
for $\,k\ge 1,\,$ we also deduce the same result when the underlined condition of (A1) is replaced by
the alternative condition of \eqref{A4}.

A proof for the nonnegativity of $\Theta_n$ under Saalsch\"utzian condition \eqref{A1} and (A2)
is given in our previous work \cite[Lemma 3.1]{CCY2}.
By tracking down the proof therein, it is not hard to find that
the underlined condition of (A2) can be replaced by
\eqref{A4} due to the same reason as above.
\end{proof}

\begin{remark} By extending \eqref{A41} inductively, we may replace \eqref{A4} by
\begin{equation}\label{A42}
\beta_1\wedge\beta_2>0, \quad \alpha_3-m\le\beta_1\wedge\beta_2\le\alpha_3-m+1,
\end{equation}
valid for any integer $\,1\le m\le n-1.\,$ In the case $\,n=1,\,$ we note
\begin{align*}
\Omega_1 =\frac{\,\beta_1\beta_2\beta_3-(1+\alpha_1)\alpha_2\alpha_3\,}
{\beta_1\beta_2\beta_3},\end{align*}
which explains in part why the last condition of (A2) is required.
The expansion \eqref{G2} is equivalent to Whipple's transformation formula \cite[4.3(4)]{Ba}.
\end{remark}

\subsection{Terminating ${}_5F_4$ hypergeometric series}
\begin{lemma}\label{lemmaA3} For each positive integer $n$, put
\begin{equation*}
\Omega_n={}_{5}F_{4} \left[\begin{array}{c} -n, n+\alpha_1, \alpha_2, \alpha_3, \alpha_4\\
\beta_1, \beta_2, \beta_{3}, \beta_4\end{array}\right]
\end{equation*}
and assume that $\,\alpha_1>-1, \,\alpha_2>0, \,\alpha_3>0,\,\alpha_4>0 $ and
\begin{equation}\label{C1}
1+\alpha_1+\alpha_2+\alpha_3 +\alpha_4 = \beta_1+\beta_2+\beta_3 +\beta_4.
\end{equation}
If parameters $\alpha_j, \beta_j$ satisfy either set of the following additional conditions simultaneously,
then $\,\Omega_n \ge0\,$ for all $\,n\ge 1.\,$
\begin{align*}
{\rm(C1)} &\left\{\begin{aligned}
&{\,\underline{\alpha_3 \le \beta_1\wedge\beta_2},}\\
&{\,\alpha_4 \le\beta_3\wedge\beta_4,}\\
&{\, \alpha_2 +\alpha_3 \le\beta_1 +\beta_2 \le 1+\alpha_1 +\alpha_3.}
\end{aligned}\right.\\
{\rm(C2)} &\left\{\begin{aligned}
&{\,\underline{\alpha_3 \le \beta_1\wedge\beta_2},}\\
&{\,\alpha_4\wedge\beta_3\wedge\beta_4\le 1+\alpha_1,}\\
&{\,1+\alpha_1 +\alpha_2\le \beta_3 +\beta_4 \le 2+\alpha_1 + \alpha_4,}\\
&{\,\alpha_2\alpha_3\alpha_4 \le \beta_1\beta_2(\beta_3+\beta_4-\alpha_1-1).}\end{aligned}\right.
\end{align*}
Moreover, the underlined condition may be replaced by \eqref{A4}.
\end{lemma}

\begin{proof} We apply Gasper's transformation formula \eqref{G1} to decompose
\begin{equation}\label{G3}
   \Omega_n  =\sum_{k=0}^n \binom{n}{k}
        \frac{(\beta_1-\alpha_3)_k (\beta_2-\alpha_3)_k(n+\alpha_1)_k(\alpha_2)_k (\alpha_4)_k}
                {(k+\beta_1 +\beta_2 -\alpha_3 -1)_{k}(\beta_1)_k(\beta_2)_k(\beta_3)_k (\beta_4)_k}\,W_k,
\end{equation}
valid under Saalsch\"utzian condition \eqref{C1}, where
$$ W_k ={}_4F_{3}\left[\begin{array}{c} -n+k, n+k+\alpha_1, k+\alpha_2, k+\alpha_4\\
        2k +\beta_1 +\beta_2 -\alpha_3, k+\beta_3, k+\beta_4\end{array}\right]
$$
with the agreement $\,W_n = 1.\,$
Since each $W_k$ is Saalsch\"utzian, if we check (A1)
of Lemma \ref{lemmaA1} with matching parameters by
\begin{align*}
(\alpha_1, \alpha_2, \alpha_3) &\mapsto (2k+\alpha_1, k+\alpha_4, k+\alpha_2),\\
(\beta_1, \beta_2, \beta_3) &\mapsto  (k+\beta_3, k+\beta_4, 2k +\beta_1 +\beta_2 -\alpha_3),
\end{align*}
then we find $\,W_k\ge 0\,$ for all $\,0\le k\le n-1\,$ when
$$\alpha_2\le \beta_1 +\beta_2 -\alpha_3\le 1+ \alpha_1, \,\,\,
\alpha_4 \le\beta_3\wedge\beta_4.$$
In view of \eqref{G3}, we may conclude $\,\Omega_n\ge 0\,$ under the assumption of (C1) in which the underlined
condition may be replaced by \eqref{A4} by the same reasoning as we employed in the proof of Lemma \ref{lemmaA1}.

To prove the positivity of $\Omega_n$ under assumptions \eqref{C1} and (C2), we apply
the transformation formula \eqref{T} to decompose
 \begin{align}\label{T1}
   \Omega_n &=\sum_{k=0}^n\binom{n}{k}
        \frac{(\sigma)_k (n+\alpha_1)_k(\alpha_4)_k}{(k+\delta)_k (\beta_3)_k(\beta_4)_k}\, U_k V_k\,,\quad\text{where}\\
        U_k &= {}_4F_3\left[\begin{array}{c} -k, k+\delta, \alpha_2, \alpha_3\\
        \sigma, \beta_1, \beta_2\end{array}\right],\nonumber\\
        V_k &= {}_4F_{3}\left[\begin{array}{c} -n+k, n+k+\alpha_1, k+\sigma,  k+\alpha_4\\
        2k+\delta+1, k+\beta_3, k+\beta_4\end{array}\right]\nonumber
\end{align}
with the convention $\,U_0 = V_n = 1.\,$ Owing to condition \eqref{C1},
we note that $U_k, V_k$ are Saalsch\"utzian for any pair $(\delta, \sigma)$ satisfying
\begin{equation}\label{B3}
\sigma = 1+\delta +\alpha_2 +\alpha_3 -\beta_1-\beta_2.
\end{equation}

By checking (A2) of Lemma \ref{lemmaA1} with matching parameters by
\begin{align*}
(\alpha_1, \alpha_2, \alpha_3) \mapsto (\delta, \alpha_2, \alpha_3),
\,\,\,(\beta_1, \beta_2, \beta_3) \mapsto  (\beta_1, \beta_2, \sigma),
\end{align*}
we find that $\,U_k\ge 0\,$ for all $\,1\le k\le n\,$ when
\begin{equation}\label{B4}
\alpha_2\le\sigma\le 2+\delta,\,\,\, \alpha_3\le\beta_1\wedge\beta_2,\,\,\,
(1+\delta)\alpha_2\alpha_3\le \beta_1\beta_2\sigma\,
\end{equation}
where the second condition may be replaced by \eqref{A4}.
On the other hand, if we check (A1) of Lemma \ref{lemmaA1} with matching parameters by
\begin{align*}
(\alpha_1, \alpha_2, \alpha_3) &\mapsto (2k+\alpha_1, k+\alpha_4, k+\sigma),\\
(\beta_1, \beta_2, \beta_3) &\mapsto  (k+\beta_3, k+\beta_4, 2k +\delta+1),
\end{align*}
then we find that $\,V_k\ge 0\,$ for all $\,0\le k\le n-1\,$ when
\begin{equation}\label{B5}
\alpha_4\le 1+\delta\le 1+\alpha_1, \,\,\, \sigma\le\beta_3\wedge\beta_4.
\end{equation}

On selecting $\,1+\delta=\alpha_4\,$ so that $\,\sigma = \beta_3 +\beta_4 -\alpha_1 -1,\,$
it is simple to find that \eqref{B3}, \eqref{B4}, \eqref{B5}
amount to (C2). In view of the decomposition \eqref{T1},
we may conclude  $\,\Omega_n\ge 0\,$ for all $\,n\ge 1\,$ under assumptions \eqref{C1}, (C2) for which
the underlined condition may be replaced by \eqref{A4}.
\end{proof}

\subsection{An optimization inequality}
In dealing with the last condition of (A2), Lemma \ref{lemmaA1},
which often causes difficulties in practice, it will be useful to invoke the following inequality.

\begin{lemma}\label{lemmaH}
Given $\,0<\alpha_1\le \alpha_2\le \alpha_3,$ we have $\,\beta_1\beta_2\beta_3\ge \alpha_1\alpha_2\alpha_3\,$
for all positive real numbers $\,\beta_1, \beta_2, \beta_3\,$ satisfying the conditions
$$ \beta_1 +\beta_2 + \beta_3 = \alpha_1 +\alpha_2 +\alpha_3, \quad
\alpha_1\le\beta_1\wedge\beta_2\wedge\beta_3\le\alpha_3.$$
\end{lemma}

\begin{proof} We may assume $\,\alpha_1<\alpha_3\,$ with no loss of generality. Setting
\begin{align*}
&\qquad f(x_1, x_2) = x_1x_2(\alpha_1+\alpha_2+\alpha_3-x_1-x_2),\\
D &= \bigl\{(x_1, x_2)\in[\alpha_1, \alpha_3]\times[\alpha_1, \alpha_3] :
\alpha_1+\alpha_2\le x_1 +x_2\le \alpha_2+\alpha_3\bigr\},
\end{align*}
it suffices to prove that $\,f(x_1, x_2)\ge \alpha_1\alpha_2\alpha_3\,$ for all $\,(x_1, x_2)\in D.$

As readily verified, $\,\nabla f(x_1, x_2) = (0, 0)\,$ only when
$$ x_1 = x_2 = \frac{\alpha_1+\alpha_2+\alpha_3}{3}$$
and the usual AM-GM inequality implies
$$f\left(\frac{\alpha_1+\alpha_2+\alpha_3}{3}, \frac{\alpha_1+\alpha_2+\alpha_3}{3}\right)
=\left(\frac{\alpha_1+\alpha_2+\alpha_3}{3}\right)^3\ge \alpha_1\alpha_2\alpha_3.$$

It remains to prove $\,f(x_1, x_2)\ge \alpha_1\alpha_2\alpha_3\,$ for all $\,(x_1, x_2)\in \partial D,\,$
where the boundary $\partial D$ consists of six line segments
\begin{align*}
\partial D = &\bigl\{\alpha_1\le x_1\le \alpha_2, \,\, x_2 = \alpha_3\bigr\} \cup
\bigl\{ \alpha_2\le x_1\le \alpha_3,\,\,x_1+x_2 = \alpha_2+\alpha_3\bigr\}\\
\cup &\bigl\{x_1=\alpha_3, \,\, \alpha_1\le x_2\le \alpha_2\bigr\} \cup
\bigl\{ \alpha_2\le x_1\le \alpha_3,\,\,x_2= \alpha_1\bigr\}\\
\cup &\bigl\{\alpha_1\le x_1\le\alpha_2, \,\, x_1+x_2 = \alpha_1+\alpha_2\bigr\} \cup
\bigl\{ x_1 = \alpha_1,\,\,\alpha_2\le x_2\le\alpha_3\bigr\}.
\end{align*}

For $\,(x_1, x_2)\in \bigl\{\alpha_1\le x_1\le \alpha_2, \,\, x_2 = \alpha_3\bigr\},\,$ the first
line segment, if we write $\,x_1 = (1-\lambda)\alpha_1 +\lambda \alpha_2,\,0\le\lambda\le 1,\,$ then
the AM-GM inequality gives
\begin{align*}
f(x_1, x_2) &= \alpha_3 x_1(\alpha_1+\alpha_2 -x_1)\\
&=\alpha_3\big[(1-\lambda)\alpha_1 +\lambda \alpha_2\big]\big[\lambda\alpha_1 +(1-\lambda) \alpha_2\big]\\
&\ge \alpha_1\alpha_2\alpha_3.
\end{align*}
On the second line segment of $\partial D$, we put
$\,x_1 = (1-\lambda)\alpha_2 +\lambda \alpha_3,\,0\le\lambda\le 1\,$ and apply the AM-GM inequality to deduce
\begin{align*}
f(x_1, x_2) &= \alpha_1 x_1(\alpha_2+\alpha_3 -x_1)\\
&=\alpha_1\big[(1-\lambda)\alpha_2 +\lambda \alpha_3\big]\big[\lambda\alpha_2 +(1-\lambda) \alpha_3\big]\\
&\ge \alpha_1\alpha_2\alpha_3.
\end{align*}
By carrying out similar reasonings for other line segments of $\partial D$, it is not difficult to
find that the same conclusion continues to hold true.
\end{proof}

\section{Hexagons and triangles}
In what follows, we shall denote by $\Phi$ the ${}_2F_3$ hypergeometric function defined in \eqref{TF}
unless specified otherwise. As it will become clearer in the sequel, our regions of positivity for $\Phi$
will be built on the basis of hexagons or triangles whose vertices consist of
permutations of $\,(a_2, \,a_1+1/2,\,2a_1).$ In order to represent such a hexagon or triangle
in such a way that each of their vertices always indicates a fixed location, as observed in Figures
\ref{Fig6}, \ref{Fig5}, it is effective to introduce the following arrangement.

\begin{definition}
For $\,a_1>0, \,a_2>0,$ we denote by $\,\xi_1, \xi_2, \xi_3\,$ the arrangement of
$\,a_2, \,a_1+1/2,\,2a_1\,$ in ascending order of magnitude so that
$$\xi_1 = \min\big( a_2, \,a_1+1/2,\,2a_1\big),\quad \xi_3 = \max\big(a_2, \,a_1+1/2,\,2a_1\big)$$
and by $\,\mathcal{A}\subset\R^3\,$ the set of all permutations of $(\xi_1, \xi_2, \xi_3)$ labelled by
\begin{align*}
\mathcal{A} = \big\{& A_1(\xi_1, \xi_2, \xi_3),\,\,\,A_2(\xi_2, \xi_1, \xi_3),\,\,\,A_3(\xi_1, \xi_3, \xi_2),\\
&A_4(\xi_3, \xi_1, \xi_2),\,\,\,A_5(\xi_2, \xi_3, \xi_1),\,\,\,A_6(\xi_3, \xi_2, \xi_1)\big\}.
\end{align*}
\end{definition}

Partly due to our way of handling those terminating hypergeometric series
considered in the preceding section, it is essential for our investigation of positivity to impose
the following {\it admissible} conditions on $(a_1, a_2)$.

\begin{definition}
We denote by $\Lambda$ the set of $\,(a_1, a_2)\in(\R_+^*)^2\,$ satisfying
\begin{equation*}
\left\{\begin{aligned}
{a_2\le \frac 12\left(3a_1 + \frac 12\right)\,\, \text{or}\,\, 1\le a_2\le 2a_1 +1\quad} &{\text{if}\quad a_1\le 1/2,}\\
{a_2 \le\max\left[a_1 +\frac32,\,\, \frac 32\left(a_1 +\frac 12\right)\right]\qquad} &{\text{if}\quad a_1\ge 1/2.}
\end{aligned}\right.
\end{equation*}
\end{definition}

Our principal result in this section reads as follows.

\begin{proposition}\label{theorem1}
For $\,(a_1, a_2)\in\Lambda,$ let $\mathbf{H}_{\mathcal{A}}$ denote the closed polygon
with vertices $\mathcal{A}$. If $\,(b_1, b_2, b_3)\in\mathbf{H}_{\mathcal{A}},\,$ then $\,\Phi(x)\ge 0\,$
for all $\,x>0\,$ with strict inequality unless $\,(b_1, b_2, b_3)\in\mathcal{A}.$
\end{proposition}

We observe that the polygon $\mathbf{H}_{\mathcal{A}}$ lies on the plane
\begin{equation}\label{H1}
b_1 + b_2 + b_3 = \xi_1+\xi_2+\xi_3 = 3a_1 + a_2 + 1/2
\end{equation}
and represents the following geometric region:
\begin{itemize}
\item If $\xi_1, \xi_2, \xi_3$ are distinct, then
$\mathbf{H}_{\mathcal{A}}$ represents the cyclic hexagon with vertices $\,A_1, A_3, A_5, A_6, A_4, A_2\,$ in the clockwise
order and center
$$\frac{\,\xi_1 +\xi_2 +\xi_3\,}{3}(1, 1, 1).$$

\item If either $\,\xi_1 =\xi_2<\xi_3\,$ or $\,\xi_1 <\xi_2=\xi_3,$ then $\mathbf{H}_{\mathcal{A}}$ represents the
triangle with vertices $\,A_1, A_3, A_4\,$ or $\,A_1, A_5, A_2\,$ in the clockwise order.

\item In the case $\,\xi_1 =\xi_2=\xi_3,$ which occurs only when $\,a_1 = 1/2,\,a_2 =1,\,$ then $\mathbf{H}_{\mathcal{A}}$
reduces to the singleton $\,\big\{(1, 1, 1)\big\}.$
\end{itemize}

Our proof is based on the expansion of $\Phi$ in the form
\begin{align}\label{H2}
\Phi(x) &=\Gamma^2\left(a_1+1/2\right)\left(\frac x2\right)^{-2a_1+1}\nonumber\\
&\qquad\times \sum_{n=0}^\infty L_n \frac{2n+2a_1-1}{n+2a_1-1} \frac{(2a_1)_n}{n!}
J^2_{n+a_1-1/2}(x),
\end{align}
valid for each point $(b_1, b_2, b_3)$ lying on the plane \eqref{H1}, where
\begin{equation}\label{H3}
L_n = {}_4F_3\left[\begin{array}{c} -n, n+ 2a_1-1, a_1+1/2, a_2\\
        b_1, b_2, b_3\end{array}\right]
\end{equation}
for each $\,n \ge 1\,$ and $\,L_0\equiv 1.\,$
This expansion results from Gasper's sums of squares formula \eqref{G},
where the free parameter $\nu$ is chosen so that each $L_n$ satisfies the Saalscht\"utzian condition, that is,
$$ 2\nu+1 = b_1 + b_2 + b_3 -a_1-a_2-1/2 = 2a_1.$$

Once each $L_n$ were shown to be nonnegative, it is obvious from \eqref{H1} that
$\,\Phi\ge 0\,$ on $(0, \infty)$. Furthermore, in view of the known fact (\cite[15.22]{Wa}) that
those positive zeros of Bessel functions $J_\nu, J_{\nu+1}$ are interlaced for any $\,\nu\in\R,\,$
we observe that $\Phi$ is strictly positive in the case when at least one of the $L_n$'s is
strictly positive. In other words, if $\,L_n\ge 0\,$ for all $\,n\ge 1,\,$ then $\Phi$ is strictly
positive unless \eqref{H1} reduces to
\begin{align}\label{H0}
\Phi(x) &=\Gamma^2\left(a_1+1/2\right)\left(\frac x2\right)^{-2a_1+1}J^2_{a_1-1/2}\left(x\right)\nonumber\\
&={}_1F_2\left[\begin{array}{c}
a_1\\a_1+1/2, \,2a_1\end{array}\biggr|\,-x^2\right],
\end{align}
which happens only when $\,(b_1, b_2, b_3)\in\mathcal{A}.$

By making use of positivity criteria established in Lemma \ref{lemmaA1}, we now proceed to prove
the nonnegativity of each $L_n$ for $\,(b_1, b_2, b_3)\in\mathbf{H}_{\mathcal{A}}\,$
and $\,(a_1, a_2)\in\Lambda$. In preparation for the proof, let us introduce
\begin{equation}\label{Q}
\mathbf{Q}_j = \big\{(b_1, b_2, b_3)\in \mathbf{H}_{\mathcal{A}}:
\xi_2\le b_j\le \xi_3\big\},\quad j=1, 2, 3.
\end{equation}
As illustrated in Figure \ref{Fig6}, if $\mathbf{H}_{\mathcal{A}}$ happens to be a hexagon, for instance,
then $\,\mathbf{Q}_1, \mathbf{Q}_2, \mathbf{Q}_3\,$ represent separately quadrilaterals
$$\,A_2 A_5 A_6 A_4, \, \,A_1 A_3 A_5 A_6, \,\,A_1 A_3 A_4 A_2.\,$$

In addition, we denote by $\mathbf{T}$ the subregion of $\mathbf{H}_{\mathcal{A}}$ bounded by
three line segments $\,A_2A_5, \, A_3 A_4, \, A_1 A_6,$ which represents the triangle with vertices
$$(\xi_1+\xi_3-\xi_2, \xi_2, \xi_2), \,\,(\xi_2, \xi_1+\xi_3-\xi_2, \xi_2),\,\,(\xi_2, \xi_2, \xi_1+\xi_3-\xi_2)$$
unless $\mathbf{H}_{\mathcal{A}}$ reduces to a point. As readily observed, the decomposition
\begin{equation}\label{D}
\mathbf{H}_{\mathcal{A}} = \mathbf{Q}_1\cup \mathbf{Q}_2\cup \mathbf{Q}_3\cup \mathbf{T}
\end{equation}
holds in any case. We note $\,\mathbf{T}\subset \mathbf{Q}_j\,$ for each $j$ only if $\,\xi_1 + \xi_3 -\xi_2\ge \xi_2.\,$

\bigskip

\begin{figure}[!ht]
 \centering
 \includegraphics[width=200pt, height= 180pt]{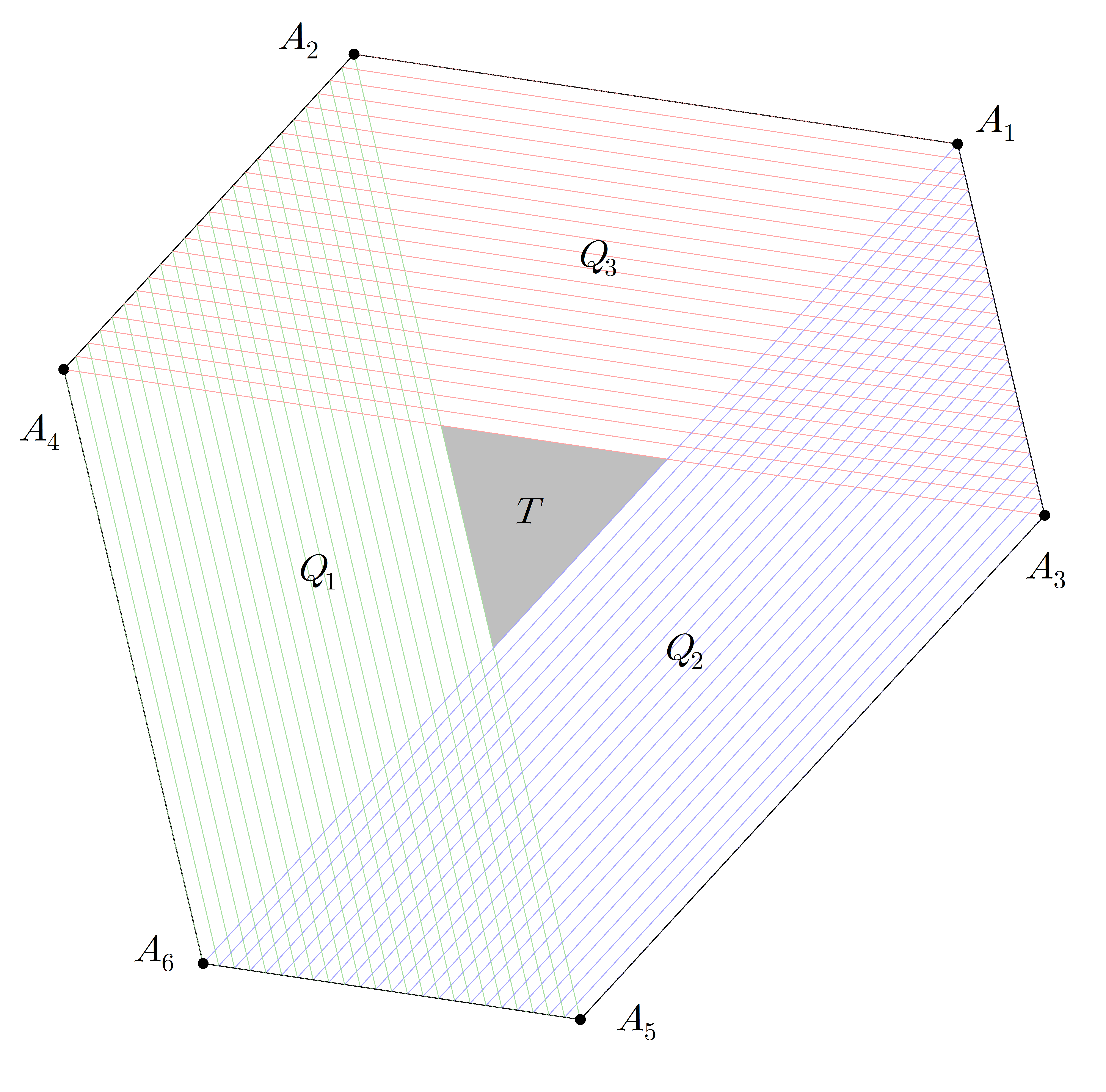}
 \caption{A decomposition of hexagon in the case $\,\xi_1 + \xi_3 > 2\xi_2.$}
\label{Fig6}
\end{figure}

Our proof will be divided into two large cases.

\paragraph{Case 1. $\,a_1 \ge 1/2.$}
We note $\,a_1 + 1/2\le 2a_1\,$ and we shall deal
with three possible occasions separately depending on the location of $a_2$.

\subparagraph{Subcase 1.1. $\,a_1+1/2\le a_2\le 2a_1.\,$}
Since $\,\xi_1+\xi_3-\xi_2 = 3a_1+1/2-a_2,$
$$\mathbf{Q}_j = \big\{(b_1, b_2, b_3)\in \mathbf{H}_{\mathcal{A}}:
a_2\le b_j\le 2a_1\big\},\quad j=1, 2, 3,$$
and $\mathbf{T}$ has vertices
$\,(c, a_2, a_2), \,(a_2, c, a_2),\,(a_2, a_2, c)\,$
with $\, c= 3a_1+1/2-a_2.$ By Lemma \ref{lemmaA1} applied with (A1), we find
$\,L_n\ge 0\,$ when
$$ a_2\le b_1\le 2a_1, \quad b_2\wedge b_3\ge a_1 + 1/2,$$
which implies $\,L_n\ge 0\,$ if $\,(b_1, b_2, b_3)\in \mathbf{Q}_1.\,$
By permuting the $b_j$'s, we find the same also holds true for
each point of $\,\mathbf{Q}_2, \mathbf{Q}_3.$

\begin{itemize}
\item[(i)] If $\,c\ge a_2,\,$ that is, $\,2a_2\le 3a_1 +1,\,$ then
$\,\mathbf{T}\subset \mathbf{Q}_j\,$ for each $j$ and we may
conclude $\,L_n\ge 0\,$ for each $\,(b_1, b_2, b_3)\in \mathbf{H}_{\mathcal{A}}.\,$

\item[(ii)] In the opposite occasion $\,2a_2> 3a_1 + 1,\,$
we apply Lemma \ref{lemmaA1} with (A1) and \eqref{A4} to deduce that $\,L_n\ge 0\,$ when
$$ a_1+1/2\le b_1\le 2a_1, \quad a_2-1\le b_2\wedge b_3\le a_2.$$
If we further restrict $\,c\ge a_2-1,\,$ that is, $\,2a_2\le 3a_1 + 3/2,\,$
then
$$\mathbf{T}\subset \big[a_1+1/2, 2a_1\big] \times \big[a_2-1, a_2\big]\times \big[a_2-1, a_2\big].$$
It follows that $L_n$ is nonnegative for each point of $\mathbf{T}$ and hence for any point of
the entire hexagon $\mathbf{H}_{\mathcal{A}}$ in this occasion.
\end{itemize}

In summary, we have proved $\,L_n\ge 0\,$ for each $\,(b_1, b_2, b_3)\in \mathbf{H}_{\mathcal{A}}\,$
when
\begin{equation}\label{H6}
a_1\ge 1/2,\quad a_1+ 1/2\le a_2\le \max\left[ 2a_1,\, \frac 32\left(a_1 + \frac 12\right)\right].
\end{equation}

\subparagraph{Subcase 1.2. $\,a_1+1/2\le 2a_1\le a_2.\,$}
We note $$\,\mathbf{H}_{\mathcal{A}}\subset
\big[a_1+1/2, a_2\big] \times \big[a_1+1/2, a_2\big] \times \big[a_1+1/2, a_2\big]\equiv I.\,$$
By Lemma \ref{lemmaA1} applied with (A2) and \eqref{A4}, we deduce $\,L_n\ge 0\,$ when
\begin{align}\label{H7}
& a_1 + 1/2\le b_1\le 2a_1 +1, \,\, a_2-1\le b_2\wedge b_3\le a_2,\nonumber\\
&\qquad 2a_1(a_1+1/2) a_2\le b_1 b_2 b_3.
\end{align}

If we further assume $\,a_2-1\le a_1 +1/2,\,$ that is, $\,a_2\le a_1 + 3/2,\,$ which holds only if
$\,1/2<a_1\le 3/2,\,$ then $\,a_2\le 2a_1 +1\,$ so that
$$I\subset \big[a_1+1/2, 2a_1+1\big] \times \big[a_2-1, a_2\big]\times \big[a_2-1, a_2\big]$$
and hence $\,L_n\ge 0\,$
provided the last condition of \eqref{H7} were verified. By applying Lemma \ref{lemmaH}
with $\,\alpha_1 = a_1 + 1/2, \,\alpha_2 = 2a_1,\,\alpha_3 = a_2\,$ and $\,\beta_j = b_j$,
it is plain to verify the last condition of \eqref{H7}.

In summary, we have proved $\,L_n\ge 0\,$ for each $\,(b_1, b_2, b_3)\in \mathbf{H}_{\mathcal{A}}\,$
when
\begin{equation}\label{H8}
1/2\le a_1\le 3/2,\quad 2a_1\le a_2\le a_1 + 3/2.
\end{equation}

\subparagraph{Subcase 1.3. $\,a_2\le a_1+1/2\le 2a_1.$}
Since $$\,\xi_1 = a_2, \,\xi_2 = a_1+1/2,\,\xi_3 = 2a_1,\quad \xi_1+\xi_3-\xi_2 = a_1+a_2-1/2\equiv c,$$
$$\mathbf{Q}_j = \big\{(b_1, b_2, b_3)\in \mathbf{H}_{\mathcal{A}}:
a_1+1/2\le b_j\le 2a_1\big\},\quad j=1, 2, 3,$$
and vertices of triangle $\mathbf{T}$ are given by
$$\,(c, a_1+1/2, a_1+1/2), \,(a_1+1/2, c, a_1+1/2),\,(a_1+1/2, a_1+1/2, c).$$
By applying Lemma \ref{lemmaA1} with (A1), we find $\,L_n\ge 0\,$ when
$$ a_1+1/2\le b_1\le 2a_1, \quad b_2\wedge b_3\ge a_2,$$
which implies $\,L_n\ge 0\,$ if $\,(b_1, b_2, b_3)\in \mathbf{Q}_1.\,$
By permuting the $b_j$'s, we find the same also holds true for
each point of $\,\mathbf{Q}_2, \mathbf{Q}_3.$

\begin{itemize}
\item[(i)] If $\,c\ge a_1+1/2,\,$ that is, $\,a_2\ge 1,\,$ then
$\,\mathbf{T}\subset \mathbf{Q}_j\,$ for all $\,j =1, 2, 3\,$ so that we may conclude
$\,L_n\ge 0\,$ for each $\,(b_1, b_2, b_3)\in \mathbf{H}_{\mathcal{A}}.\,$

\item[(ii)] In the case $\,c\le a_1+1/2\,$ or $\,a_2\le 1,\,$ we apply Lemma \ref{lemmaA1}
with (A1) and \eqref{A4} to find that $\,L_n\ge 0\,$ when
$$ a_2\le b_1\le 2a_1, \quad a_1-1/2\le b_2\wedge b_3\le a_1+1/2.$$
Since $\,a_1-1/2\le c,\,$ it is evident that
$$\mathbf{T}\subset \big[a_2, 2a_1\big] \times \big[a_1-1/2, a_1+1/2\big]\times \big[a_1-1/2, a_1+1/2\big],$$
which implies $\,L_n\ge 0\,$ for each point of $\mathbf{T}$ and hence for any point of
the entire hexagon $\mathbf{H}_{\mathcal{A}}$.
\end{itemize}

In summary, we have proved $\,L_n\ge 0\,$ for each
$\,(b_1, b_2, b_3)\in \mathbf{H}_{\mathcal{A}}\,$ when
\begin{equation}\label{H9}
a_1\ge 1/2,\quad 0<a_2\le a_1 + 1/2.
\end{equation}

\paragraph{Case 2. $\,0 < a_1 < 1/2.$}
We note $\,2a_1<a_1 + 1/2\,$ in this case and we shall deal
with three possible subcases separately depending on the location of $a_2$.
Since the method is basically same as above, our proof will be sketchy.

\subparagraph{Subcase 2.1. $\,2a_1\le a_2\le a_1+1/2.\,$}
In decomposition \eqref{D}, we note
$$\mathbf{Q}_j = \big\{(b_1, b_2, b_3)\in \mathbf{H}_{\mathcal{A}}:
a_2\le b_j\le a_1+1/2\big\},\quad j=1, 2, 3,$$
and $\mathbf{T}$ has
$\,(c, a_2, a_2), \,(a_2, c, a_2), \,(a_2, a_2, c),\,$
with $\,c\equiv 3a_1+1/2-a_2.\,$
By Lemma \ref{lemmaA1}, we find $\,L_n\ge 0\,$ when
\begin{align}\label{H10}
& a_2\le b_1\le 2a_1 +1,\,\,a_1-1/2\le b_2\wedge b_3\le a_1+1/2,\nonumber\\
&\qquad b_2\wedge b_3>0,\,\,2a_1(a_1+1/2) a_2\le b_1 b_2 b_3.
\end{align}
Since $\,2a_1+1>a_1+1/2,\,a_1-1/2<2a_1\,$ and the last condition can be verified by Lemma \ref{lemmaH},
we conclude $\,L_n\ge 0\,$ if $\,(b_1, b_2, b_3)\in \mathbf{Q}_1.\,$
By permuting the $b_j$'s, we find the same also holds for
each point of $\,\mathbf{Q}_2, \mathbf{Q}_3.$

In the case $\,c\ge a_2,\,$ that is, $\,2a_2\le 3a_1 +1,\,$ we note
$\,\mathbf{T}\subset \mathbf{Q}_j\,$ for each $j$ and hence
$\,L_n\ge 0\,$ for each $\,(b_1, b_2, b_3)\in \mathbf{H}_{\mathcal{A}}.\,$
(Unfortunately, positivity criteria of Lemma \ref{lemmaA1} are not suitable for dealing
with the case $\,c\le a_2.$)

In summary, we have proved $\,L_n\ge 0\,$ for each
$\,(b_1, b_2, b_3)\in \mathbf{H}_{\mathcal{A}}\,$ when
\begin{equation}\label{H11}
0<a_1<1/2,\quad 2a_1\le a_2\le\frac 12\left(3a_1 + \frac 12\right).
\end{equation}

\subparagraph{Subcase 2.2. $\,2a_1< a_1+1/2\le a_2.\,$}
By interchanging $\,a_1+1/2,\,a_2\,$ and proceeding as in the preceding subcase,
we find that if $\,a_2\le 2a_1 +1,\,$ then
$\,L_n\ge 0\,$ for each $\,(b_1, b_2, b_3)\in \mathbf{Q}_j,\,$
where $\mathbf{Q}_j$ denotes
$$\mathbf{Q}_j = \big\{(b_1, b_2, b_3)\in \mathbf{H}_{\mathcal{A}}:
a_1+1/2\le b_j\le a_2\big\},\quad j=1, 2, 3.$$

In regards to the triangle $\mathbf{T}$ having vertices
$$(c, a_1+1/2, a_1+1/2), \,\,(a_1+1/2, c, a_1+1/2), \,\,(a_1+1/2, a_1+1/2, c),$$
where $\,c \equiv a_1+a_2-1/2, $ if $\,c\ge a_1+1/2,$ then $\,\mathbf{T}\subset \mathbf{Q}_j\,$ for each $j$ and
$\,L_n\ge 0\,$ for $\,(b_1, b_2, b_3)\in \mathbf{H}_{\mathcal{A}}.\,$
We remark once again that positivity criteria of Lemma \ref{lemmaA1} are not suitable for dealing
with the case $\,c\le a_1+1/2.$

In summary, we have proved $\,L_n\ge 0\,$ for each
$\,(b_1, b_2, b_3)\in \mathbf{H}_{\mathcal{A}}\,$ when
\begin{equation}\label{H12}
0<a_1<1/2,\quad 1\le a_2\le 2a_1 +1.
\end{equation}

\subparagraph{Subcase 2.3. $\,a_2\le 2a_1< a_1+1/2.\,$}
As $$\,\mathbf{H}_{\mathcal{A}}\subset\big[a_2, a_1+1/2\big] \times \big[a_2, a_1+1/2\big] \times \big[a_2, a_1+1/2\big],$$
it is immediate to confirm that the positivity region obtained by \eqref{H10}
covers the whole of $\mathbf{H}_{\mathcal{A}}$. Thus $\,L_n\ge 0\,$ for each
$\,(b_1, b_2, b_3)\in \mathbf{H}_{\mathcal{A}}\,$ when
\begin{equation}\label{H13}
0<a_1<1/2,\quad 0<a_2\le 2a_1.
\end{equation}

\paragraph{The end of proof for Proposition \ref{theorem1}.}
On collecting all of the above case-by-case results summarized in
\eqref{H6}, \eqref{H8}, \eqref{H9}, \eqref{H11}, \eqref{H12}, \eqref{H13},
it is simple to find that each $L_n$ is nonnegative for
$\,(b_1, b_2, b_3)\in\mathbf{H}_{\mathcal{A}}\,$ provided
$\,(a_1, a_2)\in\Lambda.\,$ In view of the expansion \eqref{H2} and our analysis on the strict positivity,
the proof is now complete.\qed

\section{Newton polyhedra and positivity}
As standard, the Newton polyhedron of a set $\,E\subset \R_+^3\,$ refers to
the convex hull of the union of octants with corners at all points of $E$,
$$\bigcup_{\mathbf{x}\in E} \left(\mathbf{x} + \R_+^3\right),$$
which will be denoted by $\boldsymbol{\Gamma}_+(E)$ (see \cite{V} for instance).

According to the transference principle of Proposition \ref{proposition1},
if $\Phi$ were shown to be nonnegative for a single point $\,\mathbf{b} = (b_1, b_2, b_3),$
then $\Phi$ remains strictly positive for all points of octant $\,\mathbf{b} + \R_+^3\,$
except for the corner $\mathbf{b}$. As a consequence, we obtain the following
extension of Proposition \ref{theorem1} which constitutes
one of our main positivity results in this paper.

\begin{theorem}\label{theorem2}
For $\,(a_1, a_2)\in\Lambda,$ if $\,(b_1, b_2, b_3)\in\boldsymbol{\Gamma}_+\left(\mathcal{A}\right),$
then $\,\Phi(x)\ge 0\,$ for all $\,x>0\,$ and
strict inequality holds true unless $\,(b_1, b_2, b_3)\in\mathcal{A}.$
\end{theorem}

\begin{figure}[!ht]
 \centering
 \includegraphics[width=240pt, height= 240pt]{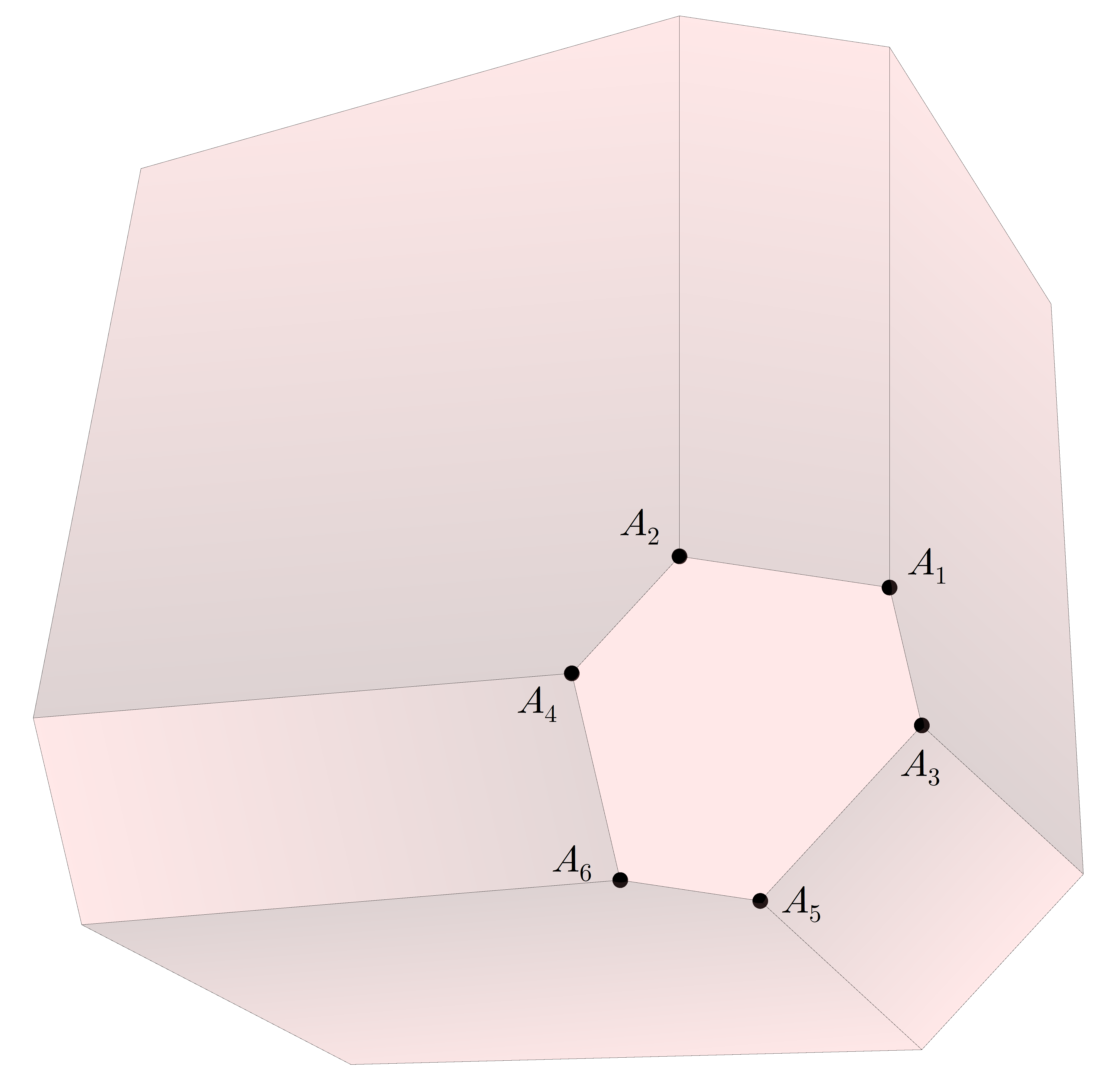}
 \caption{A Newton polyhedron with a hexagonal face.}
\label{Fig5}
\end{figure}

\begin{remark}
For practical applications, it would be desirable to characterize
$\boldsymbol{\Gamma}_+\left(\mathcal{A}\right)$ explicitly. On inspecting boundary planes
both analytically and geometrically, it is not difficult to find that
$\boldsymbol{\Gamma}_+\left(\mathcal{A}\right)$ consists of all points
$(b_1, b_2, b_3)$ satisfying the following conditions simultaneously:

\begin{equation}\label{NP1}
\left\{\begin{aligned}
&{\qquad \,b_1\wedge b_2\wedge b_3 \ge\xi_1,}\\
&{(b_1 + b_2)\wedge (b_2+b_3)\wedge (b_1+b_3)\ge \xi_1 + \xi_2,}\\
&{\qquad b_1 + b_2 + b_3\ge \xi_1 + \xi_2 +\xi_3.}\end{aligned}\right.
\end{equation}
\end{remark}

As readily observed, $\boldsymbol{\Gamma}_+\left(\mathcal{A}\right)$ has the following geometric features.

\begin{itemize}
\item If $\xi_1, \xi_2, \xi_3\,$ are distinct, as illustrated in Figure \ref{Fig5},
$\boldsymbol{\Gamma}_+\left(\mathcal{A}\right)$
represents an infinite polygonal region having 6 vertices at $\mathcal{A}$,
7 faces including the hexagon $\mathbf{H}_{\mathcal{A}}$,
and 12 edges comprised of 6 boundary lines of $\mathbf{H}_{\mathcal{A}}$ and
6 rays emanating from $\mathcal{A}$ parallel to the coordinate axes.

\item If $\,\xi_1 =\xi_2<\xi_3\,$ or $\,\xi_1<\xi_2 =\xi_3,$ 
then $\boldsymbol{\Gamma}_+\left(\mathcal{A}\right)$ represents
an infinite polygonal region having 3 vertices at $\mathcal{A}$, 4 faces including
the triangle $\mathbf{H}_{\mathcal{A}}$, and 6 edges comprised of 3 boundary lines of $\mathbf{H}_{\mathcal{A}}$ and
3 rays emanating from $\mathcal{A}$ parallel to the coordinate axes.

\item If $\,\xi_1=\xi_2 =\xi_3, $ which happens only when $\,a_1 =1/2, \,a_2 =1,\,$
then $\,\boldsymbol{\Gamma}_+\left(\mathcal{A}\right) = \mathbf{1} + \R_+^3,\,$ where $\,\mathbf{1} = (1, 1, 1).$
\end{itemize}


If $\,a_1 = a_2 =a,\,$ it is Theorem \ref{theorem2} that we wish to establish ultimately in this paper.
Since $\,(a, a)\in\Lambda\,$ for any $\,a>0,$ with no further restrictions, it may be noteworthy to single out
the result in this special occasion.

\begin{corollary}\label{corollary2}
For any $\,a>0,$ suppose
\begin{equation}\label{NP2}
\left\{\begin{aligned}
&{\quad b_1\wedge b_2\wedge b_3 \ge a,\,\,  b_1 + b_2 + b_3\ge 4a + 1/2,}\\
&{\,(b_1 + b_2)\wedge (b_2+b_3)\wedge (b_1+b_3)\ge 4a + 1/2 -\sigma,}
\end{aligned}\right.
\end{equation}
where $\,\sigma = \max\big(a+1/2,\,2a\big).$ If we define
\begin{equation*}
\phi(x) = {}_2F_3\left[\begin{array}{c}
a, a\\ b_1, b_2, b_3\end{array}\biggr| -x^2\right] \qquad(x>0),
\end{equation*}
then $\,\phi(x)\ge 0\,$ for all $\,x>0\,$ and strict inequality holds true unless
$\,(b_1, b_2, b_3)\,$ coincides with a permutation of
$\,(a,\,a+1/2,\,2a).\,$
\end{corollary}

In the case when $\,a_1, \,a_2\,$ are distinct, however, there is another Newton polyhedron
available for the positivity of $\Phi$ whose vertices are comprised of permutations of
$(a_1, a_2+1/2, 2a_2)$. To state more precisely, we shall rearrange these vertices
for the same purpose as before.

\begin{definition}
For $\,a_1>0, a_2>0,$ we denote by $\eta_1, \eta_2, \eta_3$ the arrangement of
$\,a_1, \,a_2+1/2,\,2a_2\,$ in ascending order of magnitude so that
$$\eta_1 = \min\big( a_1, \,a_2+1/2,\,2a_2\big),\quad \eta_3 = \max\big(a_1, \,a_2+1/2,\,2a_2\big)$$
and by $\,\mathcal{B}\subset\R^3\,$ the set of all permutations of $(\eta_1, \eta_2, \eta_3)$ labelled by
\begin{align*}
\mathcal{B} = \big\{& B_1(\eta_1, \eta_2, \eta_3),\,\,\,B_2(\eta_2, \eta_1, \eta_3),\,\,\,B_3(\eta_1, \eta_3, \eta_2),\\
&B_4(\eta_3, \eta_1, \eta_2),\,\,\,B_5(\eta_2, \eta_3, \eta_1),\,\,\,B_6(\eta_3, \eta_2, \eta_1)\big\}.
\end{align*}
\end{definition}

Let $\Lambda^*$ be the reflection of $\Lambda$ across the line $\,a_1=a_2,$ that is,
$$\,\Lambda^* = \left\{(a_1, a_2) : (a_2, a_1)\in\Lambda\right\}.\,$$
Applying Proposition \ref{theorem1} with roles of $\,a_1, \,a_2\,$ interchanged, we find
that \emph{ for any pair $\,(a_1, a_2)\in\Lambda^*,$ if $\,(b_1, b_2, b_3)\in\mathbf{H}_{\mathcal{B}},$
the closed polygon with vertices $\mathcal{B}$,
then $\,\Phi(x)\ge 0\,$ for all $\,x>0\,$
and strict inequality holds true unless $\,(b_1, b_2, b_3)\in\mathcal{B}.$}
Since $\,\{(a_1, a_2): 0<a_1<a_2\}\subset\Lambda^*,$ an extension of this positivity result
by the Newton polyhedron of $\mathcal{B}$ can be summarized as follows, which constitutes our
second main positivity result.

\begin{theorem}\label{theorem3}
For $\,0<a_1<a_2,$ if $\,(b_1, b_2, b_3)\in\boldsymbol{\Gamma}_+\left(\mathcal{B}\right),$
then $\,\Phi(x)\ge 0\,$ for all $\,x>0\,$ and
strict inequality holds true unless $\,(b_1, b_2, b_3)\in\mathcal{B}.$
\end{theorem}

\begin{remark}
We should emphasize that no additional conditions on the pair
$(a_1, a_2)$ are required for validity. In analogue with \eqref{NP1},
$\boldsymbol{\Gamma}_+\left(\mathcal{B}\right)$ consists of all points $(b_1, b_2, b_3)$
satisfying the following conditions simultaneously:

\begin{equation}\label{NP3}
\left\{\begin{aligned}
&{\qquad \,b_1\wedge b_2\wedge b_3 \ge\eta_1,}\\
&{(b_1 + b_2)\wedge (b_2+b_3)\wedge (b_1+b_3)\ge \eta_1 + \eta_2,}\\
&{\qquad b_1 + b_2 + b_3\ge \eta_1 + \eta_2 +\eta_3,}\end{aligned}\right.
\end{equation}
where $\,\eta_1 = a_1,\,\,\eta_1+\eta_2+\eta_3 = 3a_2 + a_1 + 1/2\,$ in any case.
\end{remark}

As both Newton polyhedra $\,\boldsymbol{\Gamma}_+\left(\mathcal{A}\right),
\boldsymbol{\Gamma}_+\left(\mathcal{B}\right)\,$ are available as positivity regions of $\Phi$,
it is natural to ask whether the Newton polyhedron $\boldsymbol{\Gamma}_+\left(\mathcal{A}\cup\mathcal{B}\right)$
also gives rise to a region of positivity. What matters are side polygonal
regions between $\,\mathbf{H}_{\mathcal{A}}, \,\mathbf{H}_{\mathcal{B}}\,$ which
we shall now investigate. Note
$$\boldsymbol{\Gamma}_+\left(\mathcal{A}\right)\cup \boldsymbol{\Gamma}_+\left(\mathcal{B}\right)
\subset \boldsymbol{\Gamma}_+\left(\mathcal{A}\cup\mathcal{B}\right).$$

\section{Side quadrilaterals}
In order to facilitate our geometric interpretations,
we shall consider both $\,\mathbf{H}_{\mathcal{A}}, \mathbf{H}_{\mathcal{B}}\,$ as if they were
hexagons and hence the interconnecting polygonal region
would consist of six side quadrilaterals. Since any of $\,\mathbf{H}_{\mathcal{A}}, \mathbf{H}_{\mathcal{B}}\,$
in other occasions may be regarded as continuous deformations of hexagons in an obvious way,
such a presumption would be harmless.

Due to our arrangements, it is simple to observe that each of six side quadrilaterals
connecting the two hexagons has vertices $\,A_i, A_j, B_j, B_i\,$ for some $\, i, j\in \{1,2, 3\}\,$
with $\,i\ne j$, if enumerated in a fixed direction, which will be denoted by
$\mathbf{S}\big[A_iA_jB_jB_i\big]$. It is convenient to interpret
\begin{equation}\label{L}
\mathbf{S}\big[A_iA_jB_jB_i\big] = \bigcup_{\lambda\in[0, 1]} L_\lambda,
\end{equation}
the collection of all line segments $L_\lambda$ joining two points
$$(1-\lambda) A_i + \lambda B_i, \quad (1-\lambda) A_j + \lambda B_j,\quad 0\le\lambda\le 1.$$

By parameterizing in an elementary way, it is immediate to characterize the line segment
$L_\lambda$ in two exemplary cases:
\begin{itemize}
\item[(i)] For $\,\mathbf{S}\big[A_1A_2B_2B_1\big]$, each
$L_\lambda$ consists of all $(b_1, b_2, b_3)$ satisfying
\begin{equation}\label{S1}
\left\{\begin{aligned}
&{ \qquad\quad b_3 = (1-\lambda)\xi_3 + \lambda \eta_3,}\\
&{\quad b_1 + b_2 = (1-\lambda)(\xi_1+\xi_2) + \lambda (\eta_1 +\eta_2),}\\
&{(1-\lambda)\xi_1 + \lambda \eta_1\le b_1\wedge b_2\le (1-\lambda)\xi_2 + \lambda \eta_2.}
\end{aligned}\right.
\end{equation}
\item[(ii)] For $\,\mathbf{S}\big[A_5A_6B_6B_5\big]$,
each $L_\lambda$ consists of all $(b_1, b_2, b_3)$ satisfying
\begin{equation}\label{S2}
\left\{\begin{aligned}
&{ \qquad\quad b_3 = (1-\lambda)\xi_1 + \lambda \eta_1,}\\
&{\quad b_1 + b_2 = (1-\lambda)(\xi_2+\xi_3) + \lambda (\eta_2 +\eta_3),}\\
&{(1-\lambda)\xi_2 + \lambda \eta_2\le b_1\wedge b_2\le (1-\lambda)\xi_3 + \lambda \eta_3.}
\end{aligned}\right.
\end{equation}
\end{itemize}

As readily verified, by characterizing each $L_\lambda$ in the same way or by geometric
considerations, it turns out that
the other four side quadrilaterals arise as mirror images of the above
two quadrilaterals across either the plane $\,b_1 = b_3\,$ or
$\,b_2 = b_3\,$ in the senses specified as follows.

\begin{lemma}\label{lemmaS}
For $\,0<a_1<a_2,$ the following hold true.
\begin{align*}
\mathbf{S}\big[A_4A_6B_6B_4\big] &= \Big\{(b_1, b_2, b_3) :
(b_3, b_2, b_1) \in \mathbf{S}\big[A_1A_2B_2B_1\big]\Big\},\\
\mathbf{S}\big[A_3A_5B_5B_3\big] &= \Big\{(b_1, b_2, b_3) :
(b_1, b_3, b_2) \in \mathbf{S}\big[A_1A_2B_2B_1\big]\Big\},\\
\mathbf{S}\big[A_1A_3B_3B_1\big] &= \Big\{(b_1, b_2, b_3) :
(b_3, b_2, b_1) \in \mathbf{S}\big[A_5A_6B_6B_5\big]\Big\},\\
\mathbf{S}\big[A_2A_4B_4B_2\big] &= \Big\{(b_1, b_2, b_3) :
(b_1, b_3, b_2) \in \mathbf{S}\big[A_5A_6B_6B_5\big]\Big\}.
\end{align*}
\end{lemma}

We are now in a position to prove the positivity of $\Phi$ when $(b_1, b_2, b_3)$
belongs to those side quadrilaterals. In view of the symmetry specified as above,
it suffices to deal with any two side quadrilaterals of type
$$\mathbf{S}\big[A_1A_2B_2B_1\big],\,\,\mathbf{S}\big[A_5A_6B_6B_5\big].$$

\subsection{Quadrilaterals of type $\mathbf{S}\big[A_5A_6B_6B_5\big].$}
As to three side quadrilaterals of this type, it turns out that {\it admissible} conditions for
$(a_1, a_2)$ are independent of $\Lambda$ and given as follows.

\begin{definition}
We denote by $\Sigma$ the set of $\,(a_1, a_2)\in(\R_+^*)^2\,$ satisfying
\begin{equation*}
\left\{\begin{aligned}
{a_1<a_2\le \min\big(2a_1,\, 1/2\big)} &{\quad\text{if}\quad a_1<1/2,}\\
{a_1<a_2\le 2a_1\quad\qquad} &{\quad\text{if}\quad a_1\ge 1/2.}
\end{aligned}\right.
\end{equation*}
\end{definition}

\begin{proposition}\label{lemmaQ1}
For $\,(a_1, a_2)\in\Sigma,$
if $\,(b_1, b_2, b_3)\in\mathbf{S}\big[A_5A_6B_6B_5\big],$
then $\,\Phi(x)\ge 0\,$ for all $\,x>0\,$ and strict inequality holds true unless it coincides
with one of vertices. By symmetry, the same holds true for
$$\mathbf{S}\big[A_1A_3B_3B_1\big],\,\,\mathbf{S}\big[A_2A_4B_4B_2\big].$$
\end{proposition}

We apply Gasper's sums of squares formula \eqref{G} to expand
\begin{align*}
\Phi(x) = \Gamma^2(\nu+1)\left(\frac x4\right)^{-2\nu}\sum_{n=0}^\infty  K_n\,\frac{2n+2\nu}{n+2\nu}
\frac{(2\nu+1)_n}{n!}J^2_{n+\nu}\left(\frac x2\right),
\end{align*}
where $\nu$ is an arbitrary real number subject to the condition that $2\nu$ does not coincide with a negative integer,
\begin{equation}
K_n = {}_5F_4\left[\begin{array}{c} -n, n+2\nu, \nu+1, a_1, a_2\\
        \nu+1/2, b_1, b_2, b_3\end{array}\right]
\end{equation}
for each $\,n\ge 1\,$ and $\,K_0\equiv 1.$ By the same reason as mentioned in the proof of Theorem \ref{theorem1},
it will be sufficient to prove $\,K_n\ge 0\,$ for all $\,n\ge 1.$

In accordance with \eqref{L}, we shall write
\begin{equation*}
\mathbf{S}\big[A_5A_6B_6B_5\big] = \bigcup_{\lambda\in[0, 1]} L_\lambda
\end{equation*}
and deduce the nonnegativity of $K_n$ on $L_\lambda$ by applying
positivity criteria for the terminating ${}_5F_4$ hypergeometric series established in
Lemma \ref{lemmaA3}.

For each $\,(b_1, b_2, b_3)\in L_\lambda,$ it follows from the characterization \eqref{S2} that
the Saalscht\"utzian condition for $K_n$ amounts to the choice
\begin{align}\label{nu}
2\nu+1 &= b_1 + b_2 + b_3 -a_1 -a_2-1/2\nonumber\\
&=2\big[(1-\lambda) a_1 + \lambda a_2\big].
\end{align}
Since $\,2\nu+1>0\,$ for any $\,\lambda\in [0, 1],$ this choice would be legitimate.

We shall divide our proof into three cases.

\paragraph{Case 1. $1/2\le a_1<a_2\le a_1+1/2 \le 2a_1.$}
Due to the correspondence
$$(\xi_1, \xi_2, \xi_3)=(a_2,\,a_1+1/2,\,2a_1),\,\,(\eta_1, \eta_2, \eta_3)=(a_1,\,a_2+1/2,\,2a_2),$$
the Saalscht\"utzian condition \eqref{nu} allows us to express
\begin{align*}
\nu+1/2 &= (1-\lambda)\xi_1 + \lambda \eta_1,\\
\nu+1 &= (1-\lambda)\xi_2 + \lambda \eta_2,\\
2\nu+1 &= (1-\lambda)\xi_3 + \lambda\eta_3.
\end{align*}

Concerning the positivity criterion (C1) of Lemma \ref{lemmaA3}, if we match
$$(\alpha_1, \alpha_2, \alpha_3, \alpha_4) \mapsto (2\nu, \xi_1, \nu+1, \eta_1),
\,\,(\beta_1, \beta_2, \beta_3, \beta_4) \mapsto (b_1, b_2, b_3, \nu+1/2), $$
then it follows readily from the above expression and \eqref{S2}
$$
\alpha_3 \le\beta_1\wedge\beta_2, \,\,
\alpha_4 \le\beta_3\wedge\beta_4, \,\,
\alpha_2 +\alpha_3 \le\beta_1 +\beta_2 \le 1+\alpha_1 +\alpha_3,
$$
whence we may conclude $\,K_n \ge 0\,$ for all $n \ge 1$.

\paragraph{Case 2. $1/2\le a_1<a_1+1/2\le a_2 \le 2a_1.$} Due to the correspondence
$$(\xi_1, \xi_2, \xi_3) = (a_1+1/2,\,a_2,\,2a_1),\,\,(\eta_1, \eta_2, \eta_3) = (a_1,\,a_2+1/2,\,2a_2),$$
the Saalscht\"utzian condition \eqref{nu} allows us to express
\begin{align*}
\nu+1/2 &= (1-\lambda)\eta_1 + \lambda\xi_2,\\
\nu+1 &= (1-\lambda)\xi_1 + \lambda\eta_2,\\
2\nu+1 &= (1-\lambda)\xi_3 + \lambda\eta_3.
\end{align*}

We shall deduce the nonnegativity of $K_n$ based on the
criterion (C2) of Lemma \ref{lemmaA3}. On matching parameters by
$$(\alpha_1, \alpha_2, \alpha_3, \alpha_4) \mapsto (2\nu, \nu+1, \eta_1,\xi_2),
\,\,(\beta_1, \beta_2, \beta_3, \beta_4) \mapsto (\nu+1/2, b_3, b_1, b_2),$$
it is straightforward to verify the first three conditions of (C2)
\begin{align*}
&\alpha_3\le \beta_1\wedge\beta_2, \,\,\,\alpha_4\wedge\beta_3\wedge\beta_4\le 1+\alpha_1,\\
& 1+\alpha_1 +\alpha_2\le \beta_3 +\beta_4 \le 2+\alpha_1 + \alpha_4.
\end{align*}
For the last condition of (C2), we use the fact $\,\xi_1-\eta_1 = \eta_2-\xi_2 =1/2 \,$
and apply the AM-GM inequality to deduce
\begin{align*}
\beta_1\beta_2 &(\beta_3+\beta_4-\alpha_1-1)\\
&= \big[(1-\lambda)\eta_1 +\lambda\xi_2\big]\big[(1-\lambda)\xi_1 +\lambda\eta_1\big]
\big[(1-\lambda)\xi_2 + \lambda\eta_2\big]\\
&\ge \big[(1-\lambda)\eta_1 +\lambda\xi_2\big]
\left\{\eta_1\xi_2 + \big[(1-\lambda)\xi_2 + \lambda\eta_1\big]/2\right\}\\
&\ge \left\{\big[(1-\lambda)\eta_1 +\lambda\xi_2\big]+ 1/2\right\} \eta_1\xi_2
= \alpha_2\alpha_3\alpha_4
 \end{align*}
and hence we may conclude $\,K_n \ge 0\,$ for all $n \ge 1$.

\paragraph{Case 3. $a_1<a_2\le 1/2, \,\,a_2\le 2a_1 \le a_1+1/2.$} We note the correspondence
$$(\xi_1, \xi_2,\xi_3)=(a_2,\,2a_1,\,a_1+1/2),\,\,(\eta_1,\eta_2,\eta_3)=(a_1,\,2a_2,\,a_2+1/2)$$
and the Saalscht\"utzian condition \eqref{nu} can be written as
\begin{align*}
\nu+1/2 &= (1-\lambda)\eta_1 + \lambda\xi_1,\\
\nu+1 &= (1-\lambda)\xi_3 + \lambda\eta_3,\\
2\nu+1 &= (1-\lambda)\xi_2 + \lambda\eta_2.
\end{align*}

In this case, we apply the positivity criterion (C1) of Lemma \ref{lemmaA3} with the alternative condition
\eqref{A4} and parameter-matchings
$$(\alpha_1, \alpha_2, \alpha_3, \alpha_4) \mapsto (2\nu, \xi_1, \nu+1, \eta_1),\,\,
(\beta_1, \beta_2, \beta_3, \beta_4) \mapsto (b_1, b_2, b_3, \nu+1/2).$$
While the other conditions can be verified easily, we use \eqref{S2} to observe
\begin{align*}
\nu\le b_1\wedge b_2 \le (1-\lambda)(a_1 + 1/2) + \lambda(a_2 + 1/2) = \nu+1,
\end{align*}
which verifies \eqref{A4} and hence $\,K_n \ge 0\,$ for all $n \ge 1$.

\medskip
Collecting all of the above case-by-case results, it is immediate to find why our arguments works for
$\Sigma$ and our proof is complete.\qed

\begin{remark}
Each of quadrilaterals
$$\mathbf{S}\big[A_5A_6B_6B_5\big],\,\,\mathbf{S}\big[A_1A_3B_3B_1\big],\,\,\mathbf{S}\big[A_2A_4B_4B_2\big]$$
lies separately on the plane
\begin{align}\label{SP1}
\left\{\begin{aligned}
&{b_1 + b_2 + b_3 = 3a_1 + \frac 12 + a_2 + \frac{\,2(a_2 - a_1)\,}{a_1 -\xi_1} (b_3 - \xi_1),}\\
&{b_1 + b_2 + b_3 = 3a_1 + \frac 12 + a_2 + \frac{\,2(a_2 - a_1)\,}{a_1 -\xi_1} (b_1 - \xi_1),}\\
&{b_1 + b_2 + b_3 = 3a_1 + \frac 12 + a_2 + \frac{\,2(a_2 - a_1)\,}{a_1 -\xi_1} (b_2 - \xi_1).}\end{aligned}\right.
\end{align}
\end{remark}

\subsection{Quadrilaterals of type $\mathbf{S}\big[A_1A_2B_2B_1\big].$}
A geometric inspection reveals that three side quadrilaterals of this type are
part of Newton polyhedra of their neighboring side quadrilaterals and hence the
positivity of $\Phi$ follows from the preceding results.

\begin{proposition}\label{lemmaQ2}
For $\,(a_1, a_2)\in\Lambda\cap\Sigma,$
if $\,(b_1, b_2, b_3)\in\mathbf{S}\big[A_1A_2B_2B_1\big],$
then $\,\Phi(x)\ge 0\,$ for all $\,x>0\,$ and strict inequality holds true unless it coincides
with one of vertices. By symmetry, the same also holds true for
$$\mathbf{S}\big[A_4A_6B_6B_4\big],\,\,\mathbf{S}\big[A_3A_5B_5B_3\big].$$
\end{proposition}

\smallskip

\begin{proof}
Let $\{\mathbf{e}_1, \mathbf{e}_2, \mathbf{e}_3\}$ be the standard basis of $\R^3$ and
$\Pi$ the union of infinite vertical strips bounded below by line segments
$A_1A_2,\,A_1 B_1,\,A_2 B_2$,
\begin{equation}\label{Q1}
\Pi = \left\{\mathbf{b} + \epsilon\,\mathbf{e}_3 : \epsilon\ge 0,
\,\mathbf{b}\in A_1A_2 \cup A_1 B_1 \cup A_2 B_2\right\}.
\end{equation}
Since $\,(a_1, a_2)\in\Lambda\cap\Sigma,$ it follows from Propositions \ref{theorem1}, \ref{lemmaQ1},
together with the transference principle, that $\Phi$ remains strictly positive for each point of $\Pi$
except for vertices $A_1, A_2, B_1, B_2$ where $\Phi$ is nonnegative.

As readily observed, $\mathbf{S}\big[A_1A_2B_2B_1\big]$ lies on the plane
\begin{align}\label{Q2}
b_1 + b_2 + b_3 = 3a_1 + 1/2 + a_2 + \frac{\,2(a_2 - a_1)\,}{\eta_3 -\xi_3} (b_3 - \xi_3).
\end{align}

\begin{itemize}
\item[(i)] If $\,\xi_1+\xi_2 = \eta_1+\eta_2\,$ or equivalently $\,\eta_3-\xi_3 = 2(a_2 -a_1),$ then
the plane becomes vertical with its equation $\,b_1+b_2 = \xi_1+\xi_2\,$ so that
$$\mathbf{S}\big[A_1A_2B_2B_1\big]\subset \Pi,$$
whence the result follows immediately.
\item[(ii)] In the case $\,\xi_1+\xi_2 <\eta_1+\eta_2,$ which may happen only when
$$\,0<a_1<a_2\le\min\big(2a_1,\,1/2\big),$$
\eqref{Q2} becomes $\,b_1+b_2 - b_3 = a_1 + a_2 -1/2\,$ and $\Pi$
consists of three vertical strips satisfying the plane equations
$$ b_1+b_2 = 2a_1 + a_2,\,\, 2b_1 + b_2 =2a_1 + 2a_2, \,\, b_1 + 2b_2 = 2a_1+2a_2.$$
By considering the inner product between normal vectors of these planes, it is simple to find,
for instance, that $\mathbf{S}\big[A_1A_2B_2B_1\big]$ and $\Pi$ are
positively distant apart with respect to the $b_1$-direction in the sense that
any point $P$ on this quadrilateral, excluding the common edges, can be written as
$\,P = Q + \epsilon\,\mathbf{e}_1\,$ with unique $\,Q\in\Pi\,$ and $\,\epsilon>0.$
Consequently, the result follows by the transference principle.
(Alternatively, one may use the inclusion
$$\mathbf{S}\big[A_1A_2B_2B_1\big]\subset\boldsymbol{\Gamma}_+
\left(\big\{A_1, A_2, A_3, A_4, B_1, B_2, B_3, B_4\big\}\right).$$
\end{itemize}
\end{proof}

\section{Main positivity result}
On combining Theorem \ref{theorem1}, Lemma \ref{lemmaS} and Propositions \ref{lemmaQ1}, \ref{lemmaQ2},
with the aid of the transference principle, we may conclude that the Newton polyhedron of
$\mathcal{A}\cup\mathcal{B}$ is a positivity region of $\Phi$ for any pair
$\,(a_1, a_2)\in \Lambda\cap\Sigma,$ what we aimed to establish ultimately.

\begin{theorem}\label{theorem4}
Suppose that $\,a_1>0,\,a_2>0\,$ and
\begin{equation}\label{M1}
\left\{\begin{aligned}
{a_1<a_2\le \min\big(2a_1,\, 1/2\big)\qquad} &{\quad\text{if}\quad a_1<1/2,}\\
{a_1<a_2\le \min\left[ 2a_1,\, \frac 32\left(a_1 + \frac 12\right)\right]} &{\quad\text{if}\quad a_1\ge 1/2.}
\end{aligned}\right.
\end{equation}
If $\,(b_1, b_2, b_3)\in\boldsymbol{\Gamma}_+\left(\mathcal{A}\cup\mathcal{B}\right),$
then $\,\Phi(x)\ge 0\,$ for all $\,x>0\,$ and
strict inequality holds true unless $\,(b_1, b_2, b_3)\in\mathcal{A}\cup\mathcal{B}.$
\end{theorem}

\begin{remark}
By exploiting the symmetry specified in Lemma \ref{lemmaS} and \eqref{SP1}, it is
not simple but elementary to give an analytic expression for the Newton polyhedron of
$\mathcal{A}\cup\mathcal{B}$ in terms of $a_1, a_2$ as follows.

\begin{itemize}
\item[(P1)] For $\,1/2\le a_1<a_2\le a_1+1/2\le 2a_1,$ $\boldsymbol{\Gamma}_+\left(\mathcal{A}\cup\mathcal{B}\right)$
consists of all points $(b_1, b_2, b_3)$ satisfying the following conditions simultaneously:
\begin{equation}\label{P1}
\left\{\begin{aligned}
&{\,b_i\ge a_1,\,\, b_i + b_j \ge a_1 + a_2 + 1/2\quad(i\ne j),}\\
&{\,b_1 + b_2 + b_3\ge 3a_1 + a_2 + 1/2,}\\
&{\,b_1 + b_2 + 3b_3\ge 3a_1 + 3a_2 + 1/2,}\\
&{\,b_1 + 3b_2 + b_3\ge 3a_1 + 3a_2 + 1/2,}\\
&{\,3b_1 + b_2 + b_3\ge 3a_1 + 3a_2 + 1/2.}\\
\end{aligned}\right.
\end{equation}

\item[(P2)] For $\,1/2\le a_1< a_1+1/2\le a_2\le 2a_1,$ $\boldsymbol{\Gamma}_+\left(\mathcal{A}\cup\mathcal{B}\right)$
consists of all points $(b_1, b_2, b_3)$ satisfying the following conditions simultaneously:
\begin{equation}\label{P2}
\left\{\begin{aligned}
&{\,b_i\ge a_1,\,\,b_i + b_j \ge a_1 + a_2 + 1/2\quad(i\ne j),}\\
&{ \, b_1 + b_2 + b_3\ge 3a_1 + a_2 + 1/2,}\\
&{\,b_1 + b_2 + b_3\ge a_1 + 3a_2 + 1/2- 4(a_2-a_1)(b_3-a_1),}\\
&{\,b_1 + b_2 + b_3\ge a_1 + 3a_2 + 1/2- 4(a_2-a_1)(b_2-a_1),}\\
&{\,b_1 + b_2 + b_3\ge a_1 + 3a_2 + 1/2- 4(a_2-a_1)(b_1-a_1).}\\
\end{aligned}\right.
\end{equation}
(Note that \eqref{P2} reduces to \eqref{P1} in the case $\,a_2 = a_1 + 1/2.$)

\item[(P3)] For $\,0<a_1<a_2\le\min\big(2a_1,\,1/2\big),$
$\boldsymbol{\Gamma}_+\left(\mathcal{A}\cup\mathcal{B}\right)$
consists of all points $(b_1, b_2, b_3)$ satisfying the following conditions simultaneously:
\begin{equation}\label{P3}
\left\{\begin{aligned}
&{\,b_i\ge a_1,\,\, b_i + b_j \ge 2a_1 + a_2,}\\
&{\,(2b_i+b_j)\wedge(b_i+ 2b_j)\ge 2a_1+ 2a_2\quad (i\ne j),}\\
&{\,b_1 + b_2 + b_3\ge 3a_1 + a_2 + 1/2,}\\
&{\,b_1 + b_2 + 3b_3\ge 3a_1 + 3a_2 + 1/2,}\\
&{\,b_1 + 3b_2 + b_3\ge 3a_1 + 3a_2 + 1/2,}\\
&{\,3b_1 + b_2 + b_3\ge 3a_1 + 3a_2 + 1/2.}\\
\end{aligned}\right.
\end{equation}
\end{itemize}
\end{remark}

\begin{figure}[!ht]
 \centering
 \includegraphics[width=250pt, height= 250pt]{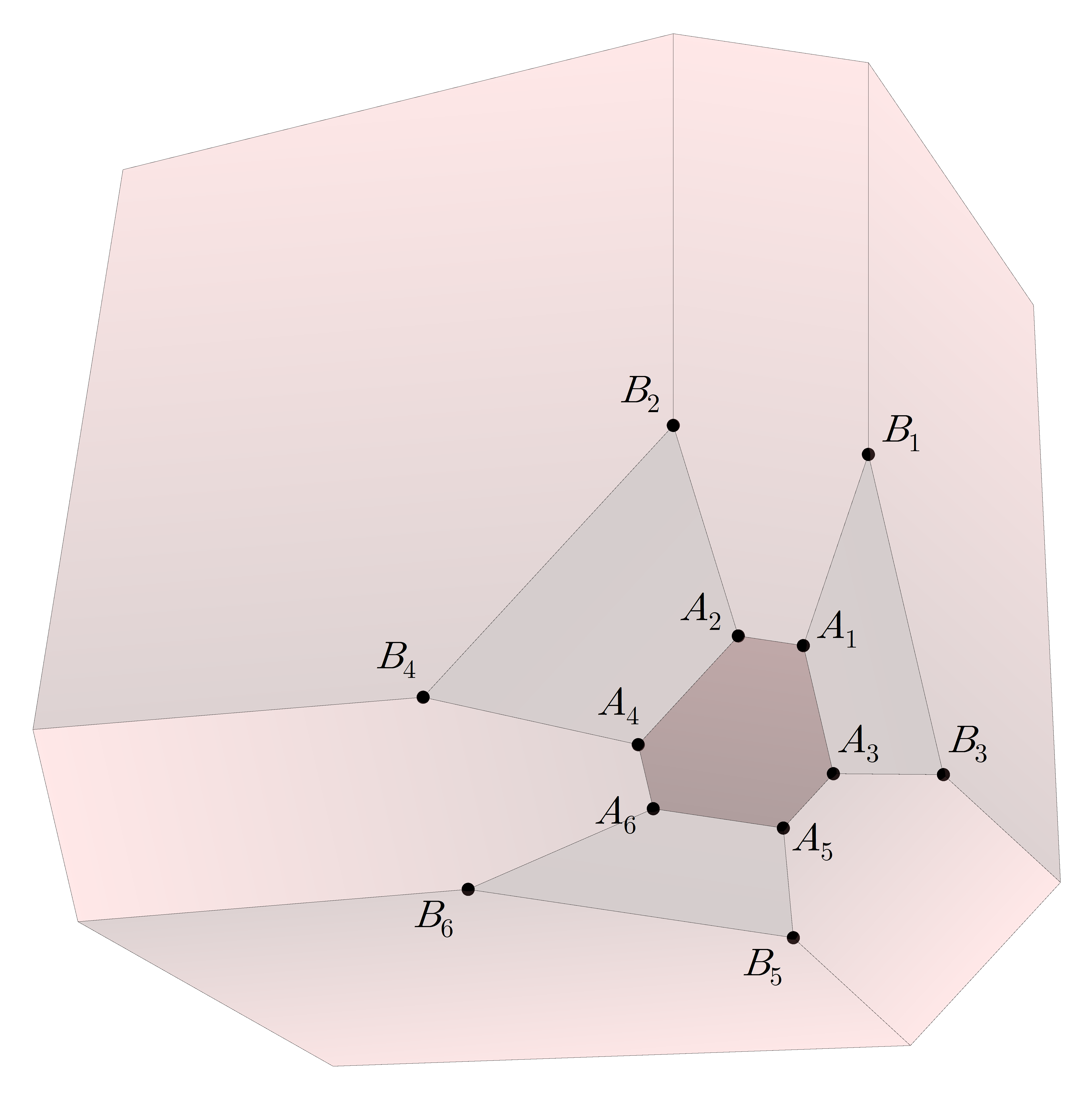}
 \caption{A Newton polyhedron of case (P1) or (P2) with a hexagonal face.}
 \label{Fig3}
\end{figure}

\begin{figure}[!ht]
 \centering
 \includegraphics[width=250pt, height= 250pt]{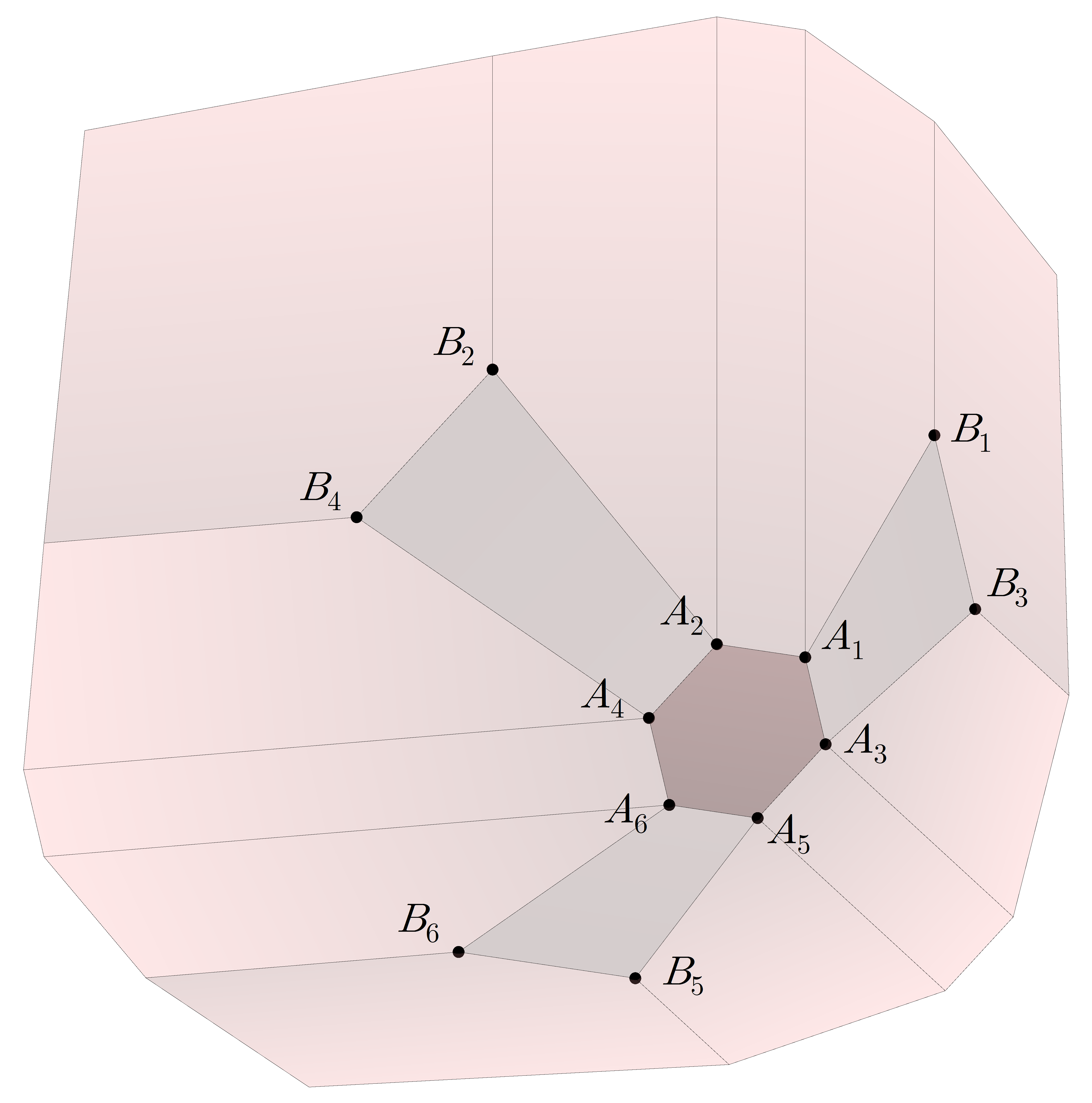}
 \caption{A Newton polyhedron of case (P3) with a hexagonal face.}
\label{Fig2}
\end{figure}

A close inspection reveals that $\,\xi_1+\xi_2 = \eta_1+\eta_2\,$ in both cases of (P1), (P2),
whereas $\,\xi_1+\xi_2 <\eta_1+\eta_2\,$ in case of (P3). From a geometric view-point,
if $\mathbf{H}_\mathcal{A}$ is a hexagon, it means that the Newton polyhedra
corresponding to (P1), (P2) are alike with the same number of faces, whereas the Newton polyhedron
corresponding to (P3) has six more faces. For convenience, we illustrate
two kinds of Newton polyhedra in Figures \ref{Fig3}, \ref{Fig2}.

We recall from \eqref{NC} of Proposition \ref{proposition2} that
\begin{equation}\label{NCR}
\mathcal{R} = \left\{\mathbf{b}\in a_1\mathbf{1} + \R_+^3:
\mathbf{b}\cdot\mathbf{1} \ge 3a_1 + a_2 + \frac 12\right\},
\end{equation}
with $\,\mathbf{1} = (1, 1, 1),$ provides a necessity region for $\Phi$. In addition,
if $\mathbf{b}$ lies outside $\mathcal{R}$, then $\Phi$ oscillates in sign at least once on $(0, \infty)$.
As it is evident that the Newton polyhedron $\boldsymbol{\Gamma}_+\left(\mathcal{A}\cup\mathcal{B}\right)$
of any case is properly contained in $\mathcal{R}$, we have thus covered all cases of interest
in the first octant of $\R^3$ except the missing region $\mathcal{R}\setminus
\boldsymbol{\Gamma}_+\left(\mathcal{A}\cup\mathcal{B}\right)$
under the assumption \eqref{M1}.

\section{Fractional integrals of Bessel functions}
This section deals with an application to the problem of positivity for
a class of Riemann-Liouville fractional integrals of Bessel functions. Of critical importance
in this subject will be the following result which improves our earlier version \cite[Theorem 3.1]{CCY1} considerably.

\begin{lemma}\label{lemmaF}
For $\,a>0, \,b>0, \,c>-1,$ we have
\begin{equation}\label{F1}
{}_2F_3\left[\begin{array}{c}
a, \,\,a+ 1/2\\
c+1, \,a+b, \, a+ b+ 1/2\end{array}\biggr| - x^2\right]\ge 0\qquad(x>0)
\end{equation}
when $a, b, c$ satisfy the following case assumptions.
\begin{itemize}
\item[\rm(i)] If $\,-1<c\le -1/2,$ then $\,0<a\le\min\big(c+1,\,\, b+c\big),\,\,b\ge 1/4.$
\item[\rm(ii)] If $\,c\ge -1/2,$ then $\,b\ge 1/4\,$ and
\begin{align*}
&\qquad 0<a\le\min\left(\frac 12,\,\, b + \frac c2 - \frac 14\right)
\quad \text{or}\\
&\frac 12\le a \le\min\left(c+1,\,\,b+c,
\,\, b+\frac c2 + \frac 14,\,\,
\frac b2  + \frac {3c}{4} + \frac 38,\,\,2b +\frac c2 - \frac 14\right).
\end{align*}
\end{itemize}
Moreover, the inequality of \eqref{F1} is strict unless
\begin{equation}\label{F2}
a = 1/2, \,b= 1, \,c = - 1/2\quad\text{or}
\quad a=1,\, b= c = 1/2.
\end{equation}
\end{lemma}

\begin{proof} In the case when $\,a\ge 1/2,$ we note that
Theorem \ref{theorem4} is applicable. More concretely, the pair $(a, a+1/2)$ falls under the
scope of (P1) and it follows from the characterization \eqref{P1} of
the associated Newton polyhedron that \eqref{F1} holds when
parameters $\,a, b, c\,$ satisfy
\begin{align}\label{F3}
\frac 12\le a \le\min &\biggl(c+1,\,\,b+c, \,\, b+\frac c2 + \frac 14,\,\,
\frac b2 + \frac {3c}{4} + \frac{3}{8},\nonumber\\
&\,\, 2b +\frac c2 - \frac 14\biggr)\quad\text{and}\quad  b\ge \frac 14.
\end{align}

On the other hand, Theorem \ref{theorem3} is applicable without any
additional condition on $a$. Since $\,(\eta_1, \eta_2, \eta_3) = (a, a+1, 2a+1),$ it follows from \eqref{NP3}
that \eqref{F1} holds when parameters $\,a, b, c\,$ satisfy
\begin{equation}\label{F4}
0<a\le\min\left(c+1,\,\, b+c,\,\,b+\frac c2 - \frac 14\right)\quad\text{and}\quad  b\ge \frac 14.
\end{equation}

On combining \eqref{F3}, \eqref{F4}, we confirm immediately validity
for the inequality \eqref{F1} under stated conditions.
For the strict positivity, it is
elementary to find that the tuple of denominator-parameters coincides with one of
permutations of $\,(a+1/2, a+1/2, 2a)\,$ or $\,(a, a+1, 2a+1)\,$ only in the case
of \eqref{F2}, whence the assertion follows by Theorems \ref{theorem3}, \ref{theorem4}.
\end{proof}

We consider the fractional integral of the form
\begin{align}\label{F5}
&\int_0^x (x-t)^{\lambda}\, t^{\mu} J_\alpha(t) dt\nonumber\\
&\quad =\,\, \frac{\,B(\lambda+1, \alpha+\mu+1)\,}{2^\alpha\,\Gamma(\alpha+1)}
\, x^{\alpha +\lambda+\mu+1}\nonumber\\
&\quad\times {}_2F_3\left[\begin{array}{c} (\alpha+\mu+1)/2, \,\,(\alpha+\mu+2)/2\\
\alpha+1,\, \left(\alpha+\lambda+\mu+2\right)/2, \,\left(\alpha+\lambda+\mu+3\right)/2\end{array}
\biggr| -\frac{\,x^2}{4}\right],
\end{align}
where $\,x>0\,$ and $\,\alpha>-1, \,\lambda>-1,\,\alpha+\mu+1>0\,$ assumed for
ensuring the convergence of integral (see \cite[13.1 (56)]{EMOT} and \cite[2.2]{Ga1}).
The problem of determining parameters $\alpha, \lambda, \mu$
for the positivity of the above integral is historic and we refer to \cite{Ga1}
for an extensive survey of earlier results contributed by numerous authors
(see also \cite{As1}, \cite{As2}).
Despite many partial results, however, it appears that a significant progress
has not been made yet for a complete description of parameters for positivity.

On recognizing that the ${}_2F_3$ hypergeometric function on the right side of \eqref{F5}
is of class \eqref{F1} with parameter matchings
$$ a= \frac{\,\alpha+\mu+1\,}{2},\,\,\, b= \frac{\,\lambda+1\,}{2},\,\,\,c = \alpha,$$
it is a matter of algebra to deduce from Lemma \ref{lemmaF} the following.

\begin{theorem}\label{theoremB}
For $\,\alpha>-1,$ the inequality
\begin{equation}\label{F6}
\int_0^x (x-t)^{\lambda}\, t^{\mu} J_\alpha(t)\, dt\, \ge \,0\qquad(x>0)
\end{equation}
holds true under the following case assumptions on $\lambda, \mu.$
\begin{itemize}
\item[\rm(i)] If $\,-1<\alpha\le -1/2,$ then $\,\lambda\ge -1/2, \,\,
-\alpha-1<\mu\le \min\big(\alpha+1,\,\, \lambda+\alpha\big).$
\item[\rm(ii)] If $\,\alpha\ge -1/2,$ then $\,\lambda\ge -1/2\,$ and
\begin{align*}
&\qquad -\alpha-1<\mu\le\min\left(-\alpha, \,\, \lambda-\frac 12\right)
\quad \text{or}\\
&-\alpha\le\mu \le\min\biggl[\alpha+1,\,\,\lambda+\alpha,
\,\, \lambda+\frac 12,\,\,2\lambda+\frac 12,\,\, \frac 12\left(\lambda+\alpha+\frac 12\right)
\biggr].
\end{align*}
\end{itemize}
Moreover, the inequality of \eqref{F6} is strict unless
\begin{equation}\label{F7}
\alpha = -1/2, \,\lambda =1, \,\mu= 1/2\quad\text{or}
\quad \alpha = 1/2, \, \lambda =0, \,\mu = 1/2
\end{equation}
and the integral in both exceptional cases reduces to
\begin{align}\label{F8}
\sqrt{\frac{2}{\pi}}\,\int_0^x (x-t)\cos t\, dt &=\sqrt{\frac{2}{\pi}}\,
\int_0^x \sin t\, dt\nonumber\\
&= 2\sqrt{\frac{2}{\pi}}\,\sin^2\left(\frac x2\right).
\end{align}
\end{theorem}

\begin{remark}
By the necessary condition \eqref{NC} for the positivity of
the ${}_2F_3$ hypergeometric function defined on the right side of \eqref{F5},
it is simple to find that the inequality \eqref{H6} fails to hold when
\begin{equation}\label{NB}
\mu>\lambda + 1/2 \quad \text{or}\quad \alpha+1<\mu\le \lambda+1/2.
\end{equation}
\end{remark}

\smallskip

A number of special cases are noteworthy. To describe, let us denote by
$\,\mathcal{P}_\alpha,\,\alpha>-1,$
the region of validity for \eqref{F6} in the $(\lambda, \mu)$-plane
determined by the above case assumptions, which represents an infinite polygonal subregion of strip
$\,[-1/2, \,\infty)\times [-\alpha-1, \,\alpha+1].$

\begin{figure}[!ht]
 \centering
 \includegraphics[width=300pt, height= 220pt]{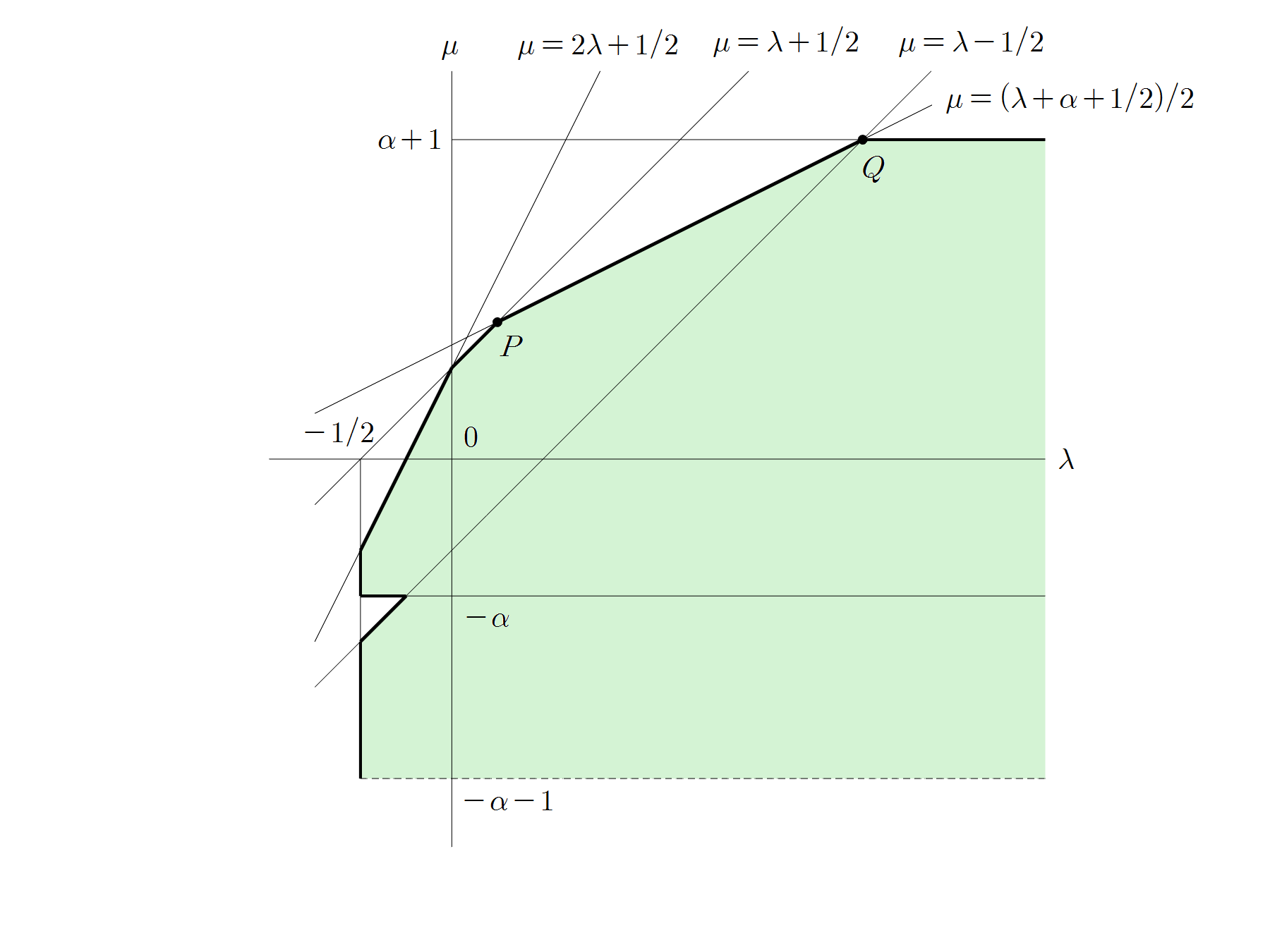}
 \caption{The validity region $\mathcal{P}_\alpha$ in the $(\lambda, \mu)$ plane for the case $\,\alpha\ge 1/2,$ where
$\,P = (\alpha-1/2, \alpha), \, Q = (\alpha+3/2, \alpha+1)$.}
\label{Fig7}
\end{figure}

As readily observed in Figure \ref{Fig7}, a part of the line $\,\mu =\lambda+1/2\,$ belongs to
$\mathcal{P}_\alpha$ only when $\,\alpha\ge 1/2\,$ and $\,0\le \lambda\le \alpha-1/2,$
which leads to

\begin{itemize}
\item[(S1)] \emph{If $\,\alpha> 1/2,\,\, 0\le \lambda\le \alpha-1/2,$ then
\begin{equation*}
\int_0^x (x-t)^{\lambda}\, t^{\lambda+1/2} J_\alpha(t)\, dt> 0 \qquad (x>0).
\end{equation*}}
\end{itemize}

This inequality was proposed by Gasper \cite[(1.5)]{Ga1} as a limiting case of
a conjecture regarding positive sums of Jacobi polynomials in Askey and Gasper \cite{AG}
and proved by himself in a rather complicated way.

In a similar manner, it is elementary to observe that a segment of
the line $\,\mu=\lambda -1/2\,$ belongs to
$\mathcal{P}_\alpha$ only when $\,\alpha\ge -1/2\,$ and
\begin{align}\label{S2}
\left\{\begin{aligned}
{-\alpha-1/2<\lambda\le \alpha + 3/2} &{\quad\text{for}\quad -1/2\le\alpha\le 0,}\\
{-1/2\le\lambda\le \alpha + 3/2} &{\quad\text{for}\quad\qquad \alpha>0.}\end{aligned}\right.
\end{align}

\begin{itemize}
\item[(S2)] \emph{For any pair $(\alpha, \lambda)$ satisfying \eqref{S2}, the inequality
\begin{equation*}
\int_0^x (x-t)^{\lambda}\, t^{\lambda-1/2} J_\alpha(t)\, dt > 0 \qquad (x>0).
\end{equation*}
holds true unless $\,\alpha = -1/2, \,\lambda =1.$}
\end{itemize}

We remark that this inequality was proved by Gasper \cite[(2.10)]{Ga1}
but only in the range
$\,1\le \lambda\le \alpha+ 3/2,\,\,\alpha>-1/2.$
For example, if we take $\,\alpha = 1/2, \,\lambda =0,$ not available in Gasper's
validity region, (S2) is equivalent to the well-known
positivity of sine integral
$$\int_0^x \frac{\sin t}{t} dt >0\qquad(x>0).$$

A line segment of $\,\mu = (\lambda+\alpha+1/2)/2\,$ belongs to
$\mathcal{P}_\alpha$ in the range
\begin{equation}\label{S3}
\left\{\begin{aligned}
&{-\alpha+1/2\le\lambda\le\alpha +3/2\quad\text{for}\quad |\alpha|\le 1/2,}\\
&{\,\,\,\,\,\alpha-1/2\le\lambda\le\alpha +3/2\quad\text{for}\quad\, \alpha\ge 1/2,}
\end{aligned}\right.
\end{equation}
where the latter case corresponds to the line segment $PQ$ in Figure \ref{Fig7},
and intersects the horizontal line $\,\mu =\alpha +\beta\,$ at $\,\lambda =
\alpha+2\beta-1/2.$ As a consequence, we find that the intersection point
$\,\big(\alpha+2\beta-1/2, \,\alpha+\beta\big)\in
\mathcal{P}_\alpha\,$
if $\, 2\alpha +2\beta \ge 1,\,\beta\le 1\,$ for $\,|\alpha|\le 1/2\,$
and $\,0\le\beta\le 1\,$ for $\,\alpha\ge 1/2,$ which leads to the following
inequality after simplifying.

\begin{itemize}
\item[(S3)] \emph{For any pair $(\alpha, \lambda)$ satisfying \eqref{S3}, the inequality
\begin{equation*}
\int_0^x (x-t)^{\lambda}\, t^{\frac 12\left(\lambda+\alpha+\frac 12\right)} J_\alpha(t)\, dt > 0 \qquad (x>0).
\end{equation*}
holds true unless $\,\alpha = -1/2, \,\lambda =1\,$ or $\,\alpha =1/2, \,\lambda = 0.$
As a particular case, if $\,\alpha+\beta\ge 1/2,\,\,0\le \beta\le 1,$ then
\begin{equation*}
\int_0^x (x-t)^{\alpha + 2\beta -1/2}\, t^{\alpha +\beta} J_\alpha(t)\, dt \ge 0\qquad(x>0)
\end{equation*}
with strict inequality unless $\,\alpha = -1/2, \,\beta =1\,$ or $\,\alpha = 1/2, \,\beta =0\,$}
\end{itemize}

The second inequality was proved by Gasper \cite{Ga1}, \cite{Ga2} (the equality case
in \cite[Theorem 7]{Ga2} must be corrected out as above). The particular case
$\,\beta =0\,$ was proved earlier by Fields and Ismail \cite{FI}.

On fixing $\,\lambda =1\,$ and inspecting the corresponding vertical line segment
contained in $\mathcal{P}_\alpha$ case by case, it is immediate to deduce the following.

\begin{itemize}
\item[(S4)] \emph{The inequality
\begin{equation*}
\int_0^x (x-t)\, t^{\mu} J_\alpha(t)\, dt>0\qquad(x>0)
\end{equation*}
holds true for any $\alpha, \mu$ satisfying the condition
\begin{equation}\label{S4}
-\alpha-1<\mu\le \min \left[\alpha+1, \,\,\frac 12\left(\alpha+\frac 32\right),\,\,
\frac 32\right],\quad \alpha>-1
\end{equation}
except for the case $\,\alpha = -1/2, \,\mu = 1/2.$}
\end{itemize}

As observed by Misiewicz and Richards \cite{MR}, this implies
\begin{equation}
\int_0^x \big(x^\delta- t^\delta\big)^\lambda\, t^{\mu} J_\alpha(t)\, dt>0\qquad(x>0)
\end{equation}
for any $\,0<\delta\le 1\le\lambda\,$ under the condition \eqref{S4} on $\alpha, \mu$,
which arises as a generalization of Kuttner's problem \cite{Ku} concerning the positivity of
Riesz means of Fourier series. The inequality under the present condition \eqref{S4} was proved by
ourselves (\cite[Theorem 4.1]{CCY1}) in a different manner and applied to extend Buhmann's classes of
radial basis functions \cite{Bu}.

In terms of Fourier cosine or sine transforms, two special cases $\,\alpha =\pm 1/2\,$
yield the following results, where we rename $\,\lambda\mapsto \alpha -1,\,\mu\mapsto \beta-1/2.\,$

\begin{itemize}
\item[(S5)]\emph{
The inequalities
\begin{align*}
{\rm(i)}\,\,\int_0^1 \cos (xt)\,(1-t)^{\alpha-1} t^{\beta-1}\,dt &\ge 0,
\qquad (\alpha, \beta)\in\Delta_c,\\
{\rm(ii)} \,\,\int_0^1 \sin (xt)\,(1-t)^{\alpha-1} t^{\beta-1}\,dt &\ge 0,
\qquad (\alpha, \beta)\in\Delta_s
\end{align*}
hold true for all $\,x>0,$ where
\begin{align*}
\Delta_c &= \Big\{(\alpha, \beta): \alpha>1, \,\, 0<\beta\le \min\left(1, \,\,\alpha -1\right)\Big\},\\
\Delta_s &= \Big\{(\alpha, \beta): \alpha\ge 1/2\quad\text{and}\quad
-1<\beta\le \min\left(0, \,\,\alpha -1\right)\\
&\qquad\qquad\text{or}\quad 0\le\beta\le\min\big[2, \,\,(\alpha+1)/2, \,\,2\alpha-1\big]\Big\}.
\end{align*}
Moreover, the first inequality is strict unless $\,\alpha =2, \,\beta=1\,$ and
the second inequality is strict unless $\,\alpha = \beta=1.$ }
\end{itemize}

We remark that $\,\Delta_c\subset\Delta_s\,$ and more generally $\mathcal{P}_\alpha$ increases
monotonically as $\alpha$ increases. Since the associated beta density
$$\,f(t) = \frac{1}{B(\alpha, \beta)} (1-t)_+^{\alpha-1}t^{\beta-1}_+\qquad(\alpha>0,\,\beta>0)$$
is neither monotone nor convex for most of $(\alpha, \beta)$,
these results indicate why it is so difficult to prove or disprove positivity of Fourier transforms
by using only intrinsic characters of densities such as convexity or monotonicity, as was suggested by Tuck \cite{T} or P\'olya \cite{P} for instance.

For negative results, we note that if $\,\beta>\alpha\,$ or
$\,\beta\le\alpha,\,\beta>1,$ then the Fourier cosine transform in (i) changes sign
infinitely often according to Proposition \ref{proposition2}.
Likewise, if $\,\beta>\alpha\,$ or $\,\beta\le\alpha,\,\beta>2,$ then the Fourier sine
transform in (ii) changes sign. In the special case $\, 0<\alpha<1,\,\,\beta =1,$
we refer to Koumandos \cite{Ko} for the estimate of those positive zeros.

The second inequality extends the result of Williamson \cite{Wi}
$$\int_0^1 \sin (xt)\,(1-t)^2 t\,dt >0,$$
which was used decisively in his simplified proof of Royall's theorem \cite{R} that
Laplace transforms of $3$-times monotone functions
on $(0, \infty)$ are univalent in the right-half plane excluding the imaginary axis.
In addition, he considered a family of Fourier sine transforms $\big(g_\alpha\big)_{\alpha>0}$
defined by
\begin{equation}\label{S5}
g_\alpha(x) = \int_0^1 \sin (xt)\,(1-t)^{\alpha-1} t\,dt\qquad(x>0)
\end{equation}
and conjectured that there exists $\,\alpha'\,$ with $\,2<\alpha'<3\,$ such that
$g_\alpha$ remains nonnegative for $\,\alpha\ge\alpha'\,$ but changes sign for $\,0<\alpha<\alpha'.$
If this conjecture were true, as he pointed out, then Royall's theorem could be
extended to the class of $\alpha$-monotone functions by the same method he employed.

\begin{proposition} Williamson's conjecture is false in that
$g_\alpha$ remains strictly positive on $(0, \infty)$ for $\,\alpha\ge 3\,$ but changes sign for
$\,0<\alpha<3.$
\end{proposition}

The proof is immediate. On taking $\,\beta = 2\,$ and inspecting $\Delta_s$,
it is simple to find that $g_\alpha$ remains strictly positive on $(0, \infty)$ for $\,\alpha\ge 3.$
On the other hand, it follows from \eqref{F5} that $\,g_\alpha(x) = c_\alpha\,x\,G_\alpha(x),$ where
\begin{equation*}
G_\alpha(x) = {}_1F_2\left[\begin{array}{c}
2\\ (\alpha+3)/2,\, (\alpha+4)/2\end{array}\biggr| -\frac{\,x^2}{4}\right]
\end{equation*}
and $c_\alpha$ is a positive constant. Setting
$$ a= 2, \,\, b= (\alpha+3)/2,\,\, c= (\alpha+4)/2,$$
it follows from the necessity part (i) of Theorem \ref{theoremN} that if $\,b+c<3a+1/2,\,$ that is, $\,\alpha<3,$
then $G_\alpha$ changes sign on $(0, \infty)$ and so does $g_\alpha$.

\bigskip

\textsc{Appendix}.
The following is a summary of our work \cite{CCY2}, \cite{CY}
concerning the positivity of ${}_1F_2$ hypergeometric functions of similar type
(see also \cite{CC}, \cite{CCY1} for relevant applications and \cite{KSW} for a probabilistic approach).

\begin{theorem}\label{theoremN}
For $\,a>0,\,b>0,\,c>0,$ put
$$\Psi(x) = {}_1F_2 \left[\begin{array}{c}
a\\ b, \,c\end{array}\biggr| - x^2 \right]\qquad(x>0).$$
\begin{itemize}
\item[\rm(i)] If $\,b\le a\,$ or $\,c\le a\,$ or $\, b+c<3a+1/2,$ then $\Psi$ changes sign.
\item[\rm(ii)] If $\,(b, c)\in P_a^*,$ then $\,\Psi(x)\ge 0\,$ for all $\,x>0\,$
and strict inequality holds true unless it belongs to
$\,S = \big\{(a+1/2, 2a),\,(2a, a+1/2)\big\},$ where
$$P_a^* =\left\{(b, c): b>a,\,\, c\ge\max\Big[ 3a+ 1/2 -b,\,\, a+\frac{a}{2(b-a)}\Big]\right\}.$$
\end{itemize}
\end{theorem}

As illustrated in Figure \ref{Fig1}, $P_a^*$ represents an infinite hyperbolic region in $\R_+^2$
containing the Newton polyhedron of $S$ characterized by
\begin{align*}
\boldsymbol{\Gamma}_+(S) = \left\{(b, c):
b\wedge c \ge \min \left(a+ 1/2,\,2a\right), \,\, b+c\ge 3a+ 1/2\right\}.
\end{align*}

\begin{figure}[!ht]
 \centering
 \includegraphics[width=300pt, height= 230pt]{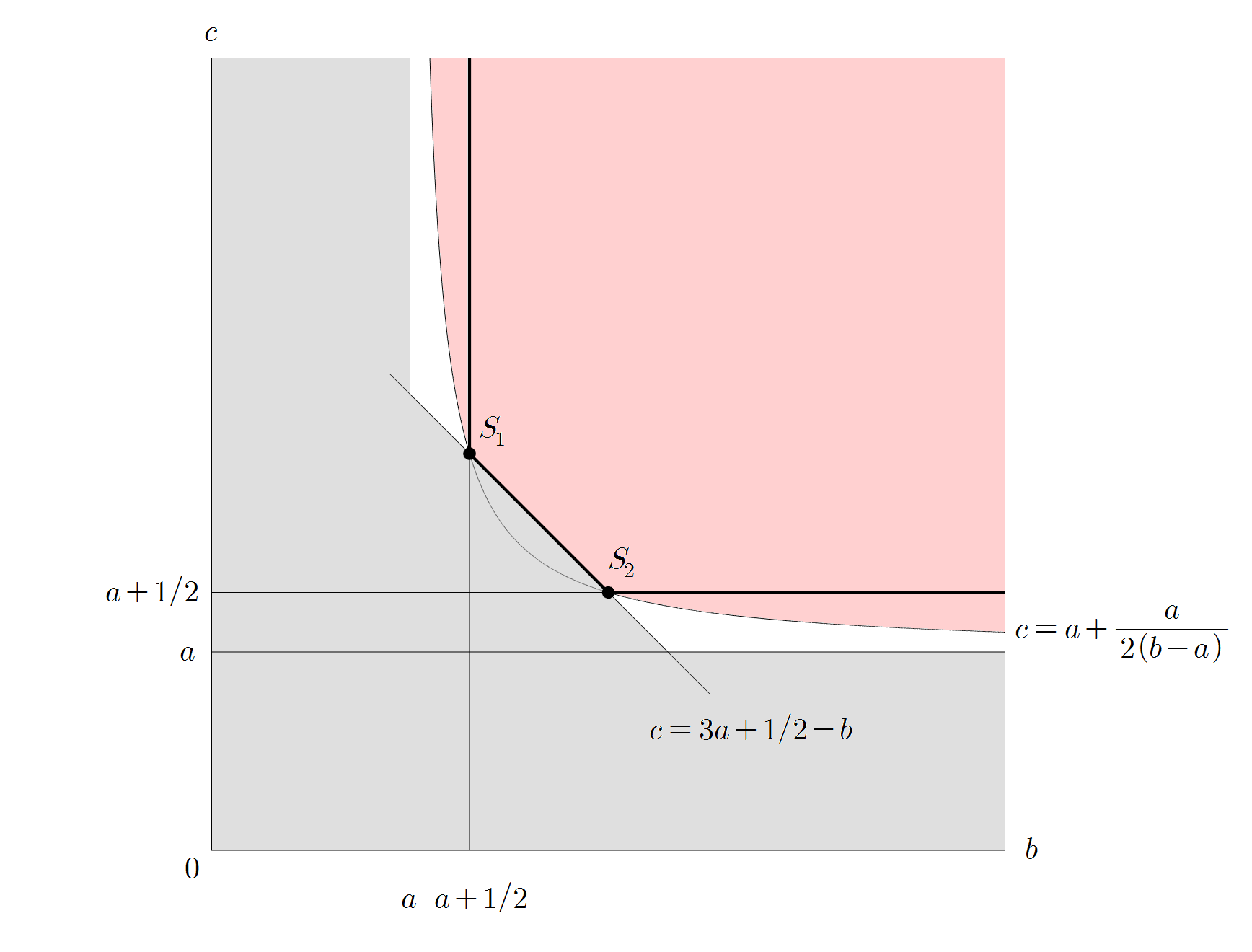}
 \caption{The positivity region $P_a^*$ (red-colored) in the case $\,a\ge 1/2,$ which contains
 the Newton polyhedron of $ S_1 = (a+1/2, 2a), \,S_2 = (2a, a+1/2).$
 If $(b, c)$ belongs to the grey-colored region, then $\Psi$ oscillates in sign.}
\label{Fig1}
\end{figure}

\bigskip

\textsc{Acknowledgements}. Yong-Kum Cho is supported by the Basic Science Research Program
through the National Research Foundation of Korea (NRF) funded by
the Ministry of Education (2018R1D1A1A09083148).
Seok-Young Chung is supported by the Chung-Ang University
Graduate Research Scholarship in 2019.


\begin{thebibliography}{12}

\bibitem[1] {As1} R. Askey,
\emph{Orthogonal Polynomials and Special Functions},
Regional Conference Series in Applied Mathematics 21, SIAM (1975)


\bibitem[2]{As2} R. Askey,
\emph{Problems which interest and/or annoy me}, J. Comput. Appl. Math. 48, pp. 3--15 (1993)


\bibitem[3]{AG} R. Askey and G. Gasper,
\emph{Positive Jacobi polynomial sums, II},
Amer. J. Math. 98, pp. 709--737 (1976)

\bibitem[4]{Ba} W. N. Bailey,
\emph{Generalized Hypergeometric Series},
Cambridge University Press, Cambridge (1935)


\bibitem[5]{Bu} M. D. Buhmann,
\emph{A new class of radial basis functions with compact support},
Math. Comput. 70, pp. 307--318 (2001)


\bibitem[6]{CC} Y.-K. Cho and S.-Y. Chung,
\emph{On the positivity and zeros of Lommel functions: Hyperbolic extension and interlacing},
J. Math. Anal. Appl. 470, pp. 898--910 (2019)


\bibitem[7]{CCY1} Y.-K. Cho, S.-Y. Chung and H. Yun,
\emph{An extension of positivity for integrals of Bessel functions and Buhmann's radial basis functions},
Proc. Amer. Math. Soc., Series B, Vol. 5, pp. 25--39 (2018)


\bibitem[8]{CCY2} Y.-K. Cho, S.-Y. Chung and H. Yun,
\emph{Rational extension of the Newton diagram for the positivity of ${}_1F_2$ hypergeometric functions
and Askey-Szeg\"o problem},
Constr. Approx. 51, pp. 49--72 (2020)

\bibitem[9]{CY} Y.-K. Cho and H. Yun,
\emph{Newton diagram of positivity for ${}_1F_2$ generalized hypergeometric functions},
Integr. Transf. Spec. F. 29, pp. 527--542 (2018)

\bibitem[10]{EMOT} A. Erd\'elyi, W. Magnus, F. Oberhettinger and F. G. Tricomi,
\emph{Tables of Integral Transforms, Vol. II}, McGraw-Hill (1954)

\bibitem[11]{FI} J. L. Fields and M. Ismail,
\emph{On the positivity of some ${}_1 F_2$'s},
SIAM J. Math. Anal. 6, pp. 551--559 (1975)


\bibitem[12]{FW} J. L. Fields and J. Wimp,
\emph{Expansions of hypergeometric functions in hypergeometric functions},
Math. Comput. 15, pp. 390--395 (1961)


\bibitem[13]{Ga1} G. Gasper,
\emph{Positive integrals of Bessel functions},
SIAM J. Math. Anal. 6, pp. 868--881 (1975)


\bibitem[14]{Ga2} G. Gasper,
\emph{Positive sums of the classical orthogonal polynomials},
SIAM J. Math. Anal. 8, pp. 423--447 (1977)

\bibitem[15]{KSW} T. Kadankova, T. Simon and M. Wang,
\emph{On some new moments of Gamma type},
Statist. Probab. Lett. 165, 108854 (2020)

\bibitem[16]{Ko} S. Koumandos,
\emph{Positive trigonometric integrals associated with some Lommel functions
of the first kind},
Mediterr. J. Math. 14:15 (2017)


\bibitem[17]{Ku} B. Kuttner,
\emph{On the Riesz means of a Fourier series, II},
J. London Math. Soc. 19, pp. 77--84 (1944)



\bibitem[18]{L} Y. L. Luke,
\emph{The Special Functions and Their Approxiations, Vol. I, II},
Academic Press, New York (1969)



\bibitem[19]{LW} Y. L. Luke and R. L. Coleman,
\emph{Expansion of hypergeometric functions in series of other hypergeometric functions},
Math. Comput. 15, 233--237 (1961)



\bibitem[20]{MR} J. Misiewicz and D. Richards,
\emph{Positivity of integrals of Bessel functions},
SIAM J. Math. Anal. 25, pp. 596--601 (1994)


\bibitem[21]{O} F. W. J. Olver, A. B. Olde Daalhuis, D. W. Lozier, B. I. Schneider, R. F. Boisvert, C. W. Clark,
B. R. Miller, B. V. Saunders, H. S. Cohl, and M. A. McClain, eds.,
\emph{NIST Digital Library of Mathematical Functions}, http://dlmf.nist.gov/, Release 1.0.27 of 2020-06-15.


\bibitem[22]{P} G. P\'olya,
\emph{\"Uber die Nullstellen gewisser ganzer Funktionen},
Math. Z. 2, pp. 352--383 (1918)


\bibitem[23]{R} N. N. Royall, Jr.,
\emph{Laplace transforms of multiply monotonic functions},
Duke Math. J. 8, pp. 546--558 (1941)

\bibitem[24]{T} E. O. Tuck,
\emph{On positivity of Fourier transforms},
Bull. Austral. Math. Soc. 74, pp. 133--138 (2006)

\bibitem[25]{V} A. Varchenko,
\emph{Newton polyhedra and estimations of oscillatory integrals},
Funct. Anal. Appl. 18, pp. 175--196 (1976)

\bibitem[26]{Wa} G. N. Watson,
\emph{A Treatise on the Theory of Bessel Functions},
Cambridge University Press, London (1995)


\bibitem[27]{Wi} R. E. Williamson,
\emph{Multiply monotone functions and their Laplace transforms},
Duke Math. J. 23, pp. 189--207 (1956)


\end{thebibliography}
\end{document}